\documentclass[a4paper,oneside]{article}
\usepackage{amsmath}
\usepackage{amsthm,amssymb,enumerate,hyperref}
\usepackage[a4paper,margin=2cm,centering,nohead,ignorefoot,footskip=5ex]{geometry}
\usepackage{url}
\usepackage{xcolor}
\hypersetup{
    colorlinks,
    linkcolor={red!50!black},
    citecolor={red!50!black},
    urlcolor={blue!80!black}
}
\newtheorem{theorem}{Theorem}[section]
\newtheorem{proposition}[theorem]{Proposition}
\newtheorem{lemma}[theorem]{Lemma}
\newtheorem{corollary}[theorem]{Corollary}
\newtheorem{definition}[theorem]{Definition}
\newtheorem{notation}[theorem]{Notation}
\newtheorem{remark}[theorem]{Remark}
\newcommand{\N}{\mathbb{N}}
\newcommand{\R}{\mathbb{R}}
\newcommand{\A}{\mathsf{A}}
\newcommand{\B}{\mathsf{B}}
\newcommand{\Id}{\mathrm{Id}}
\newcommand*{\mforall}{\tilde\forall}

\newcommand{\norm}[1]{\left\lVert#1\right\rVert}

\newcommand{\ip}[1]{\langle \, #1 \, \rangle}

\begin{document}

\title{Effective metastability for a method of alternating resolvents
}
\author{Bruno Dinis\thanks{Departamento de Matem\'atica, Faculdade de Ci\^encias da
Universidade de Lisboa, Campo Grande, Edif\'icio~C6, 1749-016~Lisboa, Portugal. \protect\url{bmdinis@fc.ul.pt}.}
\and
Pedro Pinto\thanks{Department of Mathematics, Technische Universit{\"a}t Darmstadt, Schlossgartenstra\ss{}e 7, 64289 Darmstadt, Germany.\newline	{\protect\url{pinto@mathematik.tu-darmstadt.de}}.}}

\maketitle

\begin{abstract}
A generalized method of alternating resolvents was introduced by Boikanyo and Moro{\c s}anu as a way to approximate common zeros of two maximal monotone operators. In this paper we analyse the strong convergence of this algorithm under two different sets of conditions. As a consequence we obtain effective rates of metastability (in the sense of Terence Tao) and quasi-rates of asymptotic regularity. Furthermore, we bypass the need for sequential weak compactness in the original proofs. Our quantitative results are obtained using proof-theoretical techniques in the context of the proof mining program. 
\end{abstract}
\section{Introduction}


In this paper we analyse the strong convergence of a generalized method of alternating resolvents in Hilbert spaces, introduced by Boikanyo and Moro{\c s}anu.

Let $H$ be a Hilbert space and $\A$ and $\B$ be two maximal monotone operators. Motivated by the convex feasibility problem  and the alternating projections method \cite{combettes1997hilbertian,Deutsch1985}, the \emph{method of alternating resolvents} is recursively defined as follows: $x_0 \in H$ and
\begin{equation*}
\begin{cases}
x_{2n+1} = J^{\A}_{\beta_n}(x_{2n})\\
x_{2n+2} = J^{\B}_{\mu_n}(x_{2n+1})
\end{cases}
\end{equation*}  
where $(\beta_n),(\mu_n)$ are sequences of positive real numbers. This method was shown to converge weakly to a common zero of the operators, first in \cite{BHCR(05)} for the case when $(\beta_n)$ and $(\mu_n)$ are constant and equal, and later in \cite{BM(AM11)} for the general case (also including error terms). In order to obtain strong convergence, the method of alternating resolvents was generalized by Boikanyo and Moro{\c s}anu in \cite{BM(13)} in the following way.
For $n \in \N$ we define
\begin{equation}\label{HPPA2e}\tag{\textsf{MAR}}
\begin{cases}
x_{2n+1}&=J_{\beta_n}^{\A} \left( \alpha_n u+(1-\alpha_n) x_{2n}+e_n\right)\\[1mm] 
x_{2n+2}&=J_{\mu_n}^{\B} \left( \lambda_n u+(1-\lambda_n) x_{2n+1}+e'_n\right)
\end{cases}
\end{equation}
where $u,x_0 \in H$ are given, $(\alpha_n), (\lambda_n) \subset (0,1) $, $(\beta_n),(\mu_n) \subset (0,+\infty)$, and $(e_n)$ and $(e'_n)$ are sequences of errors.
We denote by \eqref{HPPA2} the exact counterpart of the algorithm \eqref{HPPA2e}, i.e. without error terms, which is defined by
\begin{equation}\label{HPPA2}\tag{\textsf{MAR$^\star$}}
\begin{cases}
y_{2n+1}&=J_{\beta_n}^{\A} \left( \alpha_n u+(1-\alpha_n) y_{2n}\right)\\ 
y_{2n+2}&=J_{\mu_n}^{\B} \left( \lambda_n u+(1-\lambda_n) y_{2n+1}\right)
\end{cases}
\end{equation}
where $y_0=x_0$. 

Motivated by the success of the \emph{Halpern iterations} in fixed point theory \cite{Halpern67}, the \emph{Halpern-type proximal point algorithm} ($\mathsf{HPPA}$) 
\begin{equation}\label{HPPA} \tag{\textsf{HPPA}}
x_{n+1}=\alpha_n u+(1-\alpha_n)J_{\beta_n}(x_n)
\end{equation}
%
was considered as a way to upgrade the weak convergence of the proximal point algorithm to strong convergence (see e.g.\ \cite{BM(11),KT(00),X(02)}). The generalized method of alternating resolvents can be seen as form of Halpern-type proximal point algorithm generalized to two operators in an alternating fashion (cf. Section~\ref{s:connecting} -- see also \cite{LN(17)}). 

Boikanyo and Moro\c{s}ano showed in \cite{BM(13)} that \eqref{HPPA2e} is strongly convergent under the following mild assumptions. 
\begin{enumerate}[($C_1$)]
\item $\lim \alpha_n =0$ 
\item $\lim \lambda_n =0$
\item either $\sum_{n=0}^{\infty} \alpha_n = \infty$ or $\sum_{n=0}^{\infty} \lambda_n = \infty$
\item $(\beta_n)$ is bounded away from zero and such that $\lim \left( 1- \frac{\beta_{n+1}}{\beta_n}\right)=0 $
\item  $(\mu_n)$ is bounded away from zero and such that $ \lim \left( 1- \frac{\mu_{n+1}}{\mu_n}\right)=0.$
\end{enumerate} 
\begin{theorem}[{\cite[Theorem~3.2]{BM(13)}}]\label{t:BM_exact}
Let $\A:D(\A) \subset H \to 2^H$ and $\B:D(\B) \subset H \to 2^H$ be maximal monotone operators such that $S:= \A^{-1}(0) \cap \B^{-1}(0) \neq \emptyset$. For $x_0, u \in H $, let $(x_n)$ be generated by \eqref{HPPA2e}. Assume that $(C_1 )-(C_5)$ hold. If $\sum_{n=0}^{\infty} \norm{e_n} < \infty$ or $\sum_{n=0}^{\infty} \norm{e'_n} < \infty$, then $(x_n)$ converges strongly to the projection point $u$ onto $S$.
\end{theorem}
 The same authors revisited this result, eliminating the conditions $(C_4)$ and $(C_5)$, in a follow up paper \cite[Theorem~8]{BM(12)}. This improvement results from a different strategy in the proof: it now relies on a discussion by cases depending on whether a certain auxiliary sequence is strictly increasing or not, and on a result due to Maingé \cite{M(08)}. The original results allow for several alternative conditions on the error terms. However, the proofs focus first on the exact iteration and consider the convergence with error terms only afterwards. In this regard, our analyses follow the same strategy. The particular condition on the error terms is not central in our analyses, nevertheless we comment on the other possibilities in Remark~\ref{r:errors}.

 Boikanyo and Moro\c{s}ano's results are convergence statements for a certain sequence $(x_n)$. For the purpose of a quantitative analysis it is better to look at the equivalent property of being a Cauchy sequence, i.e.\ 
 \begin{equation*}
\forall k \in \N \,\exists n \, \forall i,j \geq n  \left(\norm{x_i-x_j}\leq \frac{1}{k+1} \right).
\end{equation*}
In general it is not possible to obtain computable information on the value of $n$ (see e.g.\ \cite{N(15)} for details on this matter). Instead one turns to the equivalent finitary version of the Cauchy property, which has been called \emph{metastability}  \cite{T(08b),T(08a)}, i.e.\ 
 \begin{equation}\label{metastab}
\forall k \in \N \,\forall f : \N \to \N \,\exists n \, \forall i,j \in [n,n+f(n)] \left(\norm{x_i-x_j}\leq \frac{1}{k+1} \right).
\end{equation}
For the statement \eqref{metastab}, we obtain a highly uniform computable rate of metastability, i.e.\  
a computable functional $\mu: \N \times \N^{\N}\to \N $ such that
 \begin{equation}\label{metastabb}
\forall k \in \N \,\forall f : \N \to \N \,\exists n \leq \mu(k,f) \, \forall i,j \in [n,n+f(n)] \left(\norm{x_i-x_j}\leq \frac{1}{k+1} \right).
\end{equation}
Note that a computable rate for \eqref{metastab} does not entail computable information for the equivalent Cauchy property. Indeed, since the argument is by contradiction, this equivalence is non-effective and thus one cannot construct a bound for the Cauchy property from a rate of metastability. Nevertheless, results on metastability increase one's knowledge on the behaviour of the iteration and may allow for a deeper understanding of these types of algorithms. This idea has played a significant role in several recent results, see for example \cite{AGT(10),GT(08),KL(09),T(08b)}. 
 
The methods used in this paper are set in the framework of proof mining \cite{K(08),K(18)}, a program that describes the process of using proof-theoretical techniques to analyse mathematical proofs with the aim of extracting new information. That being said, our results and proofs do not presuppose any particular knowledge of logical tools because the latter are only used as an intermediate step and are not visible in the final product.

Apart from obtaining effective quantitative information and similarly to \cite{DP(21),DP(20), DP(ta),PP(ta)}, applying a technique developed in \cite{FFLLPP(19)}, our results are established without the sequential weak compactness arguments required in the original proofs and only rely on a weak form of projection. This fact is reflected in the complexity of the bounds obtained, which are primitive recursive (in the sense of Kleene). For other proof mining results on the convergence of algorithms based on the proximal point algorithm see e.g.\ \cite{Koh(ta),Koh(ta2),LLPP(ta),LS(18)}.

The structure of the paper is the following. In Section~\ref{sectionPrelim} we recall the relevant terminology as well as some well-known results from the theory of monotone operators in Hilbert spaces. We also present some lemmas necessary for our analysis. In Section~\ref{s:3} we obtain an effective metastability bound on the algorithm \eqref{HPPA2} under appropriate quantitative versions of conditions $(C_1) - (C_5)$. We also establish a connection between \eqref{HPPA2} and a generalization of \eqref{HPPA} for two operators.  In Section~\ref{s:second} a rate of metastability for \eqref{HPPA2} is obtained without the conditions $(C_4)$ and $(C_5)$. Section~\ref{s:Reduction} extends the results  from Sections~\ref{s:3} and \ref{s:second} to the iteration \eqref{HPPA2e}. Some final remarks are left  to the last section.

\section{Preliminaries}\label{sectionPrelim}

\subsection{Monotone Operators and resolvent functions}
Throughout we let $H$ be a real Hilbert space with inner product $\ip{\cdot, \cdot}$ and norm $\norm{\cdot}$.
\begin{definition}
A mapping $T : H \to H$ is called \emph{nonexpansive} if 
$$\forall x, y \in H \left(\norm{T (x) -T (y)}\leq \norm{x -y}\right),$$ 
and \emph{firmly nonexpansive} if
$$\forall x, y \in H\left( \norm{T(x)-T(y)}^2 \leq \norm{x-y}^2-\norm{(\Id -T)(x)- (\Id-T)(y)}^2 \right).$$
\end{definition}
Note that any firmly nonexpansive mapping is also nonexpansive. If $T$ is nonexpansive, then the set of its fixed points $\{x \in H: T(x)=x\}$ is a closed and convex subset of $H$. 
 We recall that an operator $\A:H \to 2^{H}$  is said to be \emph{monotone} if and only if whenever $(x,y)$ and $(x',y')$ are elements of the graph of $\A$, it holds that $\ip{ x-x',y-y'} \geq 0$. A monotone operator $\A$ is said to be \emph{maximal monotone} if  the graph of $\A$ is not properly contained in the graph of any other monotone operator on $H$. For every positive real number $\gamma$, we use $J^\A_\gamma$  to denote the single-valued \emph{resolvent function} of $\A$  defined by $J^\A_\gamma = (I + \gamma\A )^{-1}.$ The resolvent functions are firmly nonexpansive and their fixed points coincide with the zeros of the operator.

The following lemmas are well-known.
\begin{lemma}[Resolvent identity]
For $a,b>0$, the following identity holds for every $x \in H$
\begin{equation*}
J^\A_a (x) = J^\A_b\left(\frac{b}{a}x+ \left(1-\frac{b}{a} \right)J^\A_a (x)\right).
\end{equation*}
\end{lemma}
\begin{lemma}[\cite{MX(04)}]\label{lemmaresolvineq}
If $0<a\leq b$, then $\norm{J^\A_a (x)-x} \leq 2 \norm{J^\A_b (x)-x}$, for all $x \in H$.
\end{lemma}
\begin{lemma}[\cite{GK(90)}]\label{l:fnin}
A mapping $T$ is firmly nonexpansive if and only if the mapping $2T-\Id$ is nonexpansive. 
\end{lemma}
For a comprehensive introduction to convex analysis and the theory of monotone operators in Hilbert spaces we refer to \cite{BC(17)}.

We fix $\A,\B$ maximal monotone operators on $H$ and assume henceforth the set $S:=\A^{-1}(0) \cap \B^{-1}(0)$ of the common \emph{zeros} of $\A$ and $\B$ to be nonempty.

\subsection{Quantitative notions}

Consider the strong majorizability relation $\leq^{\ast}$ from \cite{bezem1985strongly} for functions $f,g:\N\to\N$
$$g\leq^* f := \forall n, m\in\N\,\left(m\leq n\to \left( g(m)\leq f(n) \land f(m)\leq f(n)\right)\right).$$
If $f\leq^* f$ we say that $f$ is \emph{monotone}. This corresponds to saying that $f$ is a  nondecreasing function. Universal quantifications over monotone functions are denoted $\mforall f \, (\dots)$. A functional $\varphi:\N\times \N^{\N}\to\N$ is \emph{monotone} if for all $m,n\in\N$ and all $f,g:\N\to\N$,
$$\left(m\leq n \land g\leq^* f\right) \to \left(\varphi(m,g)\leq \varphi(n,f)\right).$$
Functions depending on several variables (ranging over $\N$ or $\N^\N$), and taking values in $\N$, are said to be monotone if they are monotone in all the variables.
By the theoretical results underlying the extractions in this paper we always obtain bounds given by monotone functions.
\begin{definition}\label{d:RN}
Let $(a_n)$ be a sequence of real numbers. 
\begin{enumerate}[$(i)$]
\item A \emph{rate of convergence} for 
$a_n \to 0$ is a function $\gamma:\N\to\N$ such that
\[\forall k\in\N\, \forall n\geq \gamma(k)\,\left(|a_n|\leq \frac1{k+1}\right).\]
\item A \emph{quasi-rate of convergence} for $a_n \to 0$ is a functional $\Gamma:\N \times \N^{\N}\to\N$ such that
\[\forall k\in\N\, \forall f:\N \to \N \, \exists n\leq \Gamma(k,f)\,\forall m \in [n,f(n)] \left(|a_m|\leq \frac1{k+1}\right).\]
\item A \emph{rate for $\limsup a_n \leq 0$} is a function $\gamma:\N\to\N$ such that
\[\forall k\in\N\, \forall n\geq \gamma(k)\,\left(a_n\leq \frac1{k+1}\right).\]
\item A \emph{rate of divergence} for $\sum a_n = \infty$ is a function $\gamma:\N\to\N$ such that
\[\forall k,n \in\N \left(\sum_{i=0}^{\gamma(k)+n} a_i \geq k\right).\]
\item A \emph{Cauchy rate} for $\sum a_n < \infty$ is a function $\gamma:\N\to\N$ such that
\[\forall k,n \in\N \left(\sum_{i=\gamma(k)+1}^{\gamma(k)+n} a_i \leq \frac{1}{k+1}\right).\]
\end{enumerate}
\end{definition}
\begin{definition}\label{d:HS}
Let $(x_n)$ be a sequence in $H$ and $x \in H$.
\begin{enumerate}[$(i)$]
\item A \emph{rate of convergence} for $x_n \to x$ is a rate of convergence for $\norm{x_n-x} \to 0$. 
\item A \emph{quasi-rate of convergence} for $x_n \to x$ is a quasi-rate of convergence  for $\norm{x_n-x} \to 0$. 
\item A \emph{Cauchy rate} for $(x_n)$ is a function $\gamma:\N\to\N$ such that
\[\forall k \in\N \, \forall i,j \geq \gamma(k)\left(\norm{x_i-x_j}\leq \frac{1}{k+1}\right).\]
\item A \emph{rate of metastability} is a functional $\Gamma:\N \times \N^{\N}\to\N$ such that
 \begin{equation*}
\forall k \in \N \,\forall f : \N \to \N \,\exists n \leq \Gamma(k,f) \, \forall i,j \in [n,f(n)] \left(\norm{x_i-x_j}\leq \frac{1}{k+1} \right).
\end{equation*}
\end{enumerate}
\end{definition}
\begin{definition}
Let $(x_n)$ be a sequence in $H$ and consider mappings $T,T': H \to H$. 
\begin{enumerate}[$(i)$]
\item The sequence $(x_n)$ is \emph{asymptotically regular} w.r.t. $T$ if $\norm{T(x_{n})-x_n} \to 0$. A \emph{(quasi-)rate of asymptotic regularity} for $(x_n)$ w.r.t.\ $T$  is a (quasi-)rate of convergence for $\norm{T(x_{n})-x_n} \to 0$.
\item A \emph{(quasi-)rate of asymptotic regularity} for $(x_n)$ w.r.t. $T$ and $T'$  is a (quasi-)rate of convergence for $\max\{\norm{T(x_{n})-x_n}, \norm{T'(x_{n})-x_n}\} \to 0$.
\end{enumerate}
\end{definition}
\begin{remark}\label{r:maj}
~
\begin{enumerate}[$(i)$]
\item The definition of quasi-rate of asymptotic regularity w.r.t.\ $T$ and $T'$ entails that 
\begin{equation*}
\forall k \in \N \, \forall f:\N \to \N \, \exists n \leq \Gamma(k,f)\, \forall m \in [n,f(n)] \left(\norm{T(x_{m})-x_m} \leq \frac{1}{k+1} \wedge \norm{T'(x_{m})-x_m} \leq \frac{1}{k+1} \right).
\end{equation*}
\item For the sake of simplicity we use the interval $[n,f(n)]$ when talking about metastability, instead of the usual notion \eqref{metastab}. One can recover the usual form by considering the function  $n \mapsto n+f(n)$.
\item Whenever we write $\mforall f \,(\dots)$ we restrict our arguments to monotone functions in $\N^\N$. There is no loss in generality in doing so, as for a function $f: \N \to \N$ one has $f \leq^* \! f^{\mathrm{maj}}$, where $f^{\mathrm{maj}}$ is the monotone function defined by $f^{\mathrm{maj}}(n):= \max\{f(i)\, :\, i \leq n\}$. In this way, we avoid constantly having to switch from $f$ to  $f^{\mathrm{maj}}$, and simplify the notation. For example, one may assume that a rate of convergence is monotone.
\end{enumerate}
\end{remark}
\begin{notation}\label{n:ftilde}
\begin{enumerate}
 \item Given $M \in \N$ and $f:\N \to \N$, we denote by $g_f$ the function defined by $g_f(m):=f(2m+1)$ and by $\widetilde{f}[M]$ the function defined by $\widetilde{f}[M](m):=f(\max\{m,M\})$.
\item Consider a function $\varphi$ on tuples of variables $\bar{x}$, $\bar{y}$. If we wish to consider the variables $\bar{x}$ as parameters we write $\varphi[\bar{x}](\bar{y})$. For simplicity of notation we may then even omit the parameters and simply write $\varphi(\bar{y})$.
\end{enumerate}
\end{notation}
The next result concerns the conjunction of two rates of metastability in an abstract way.
\begin{proposition}\label{p:metameta}
	Let $X$ be a set, $A(k,m,x)$ and $B(k,m,x)$ be formulas with parameters $k,m,x$,  and $\phi_1,\phi_2$ be monotone functions satisfying 
	\begin{itemize}
		\item[(i)] $\forall k \in \N\, \mforall f\in \N^\N \,\exists n\leq \phi_1(k,f)\, \exists x \in X\,\forall m\in[n,f(n)]\, A(k,m,x)$ 
		\item[(ii)] $\forall k \in \N\, \mforall f\in \N^\N\, \exists n\leq \phi_2(k,f)\, \exists x \in X\,\forall m\in[n,f(n)]\, B(k,m,x)$.
	\end{itemize}
	Then
	\[\forall k\in \N\, \mforall f \in \N^\N \,\exists n\leq \overline{\Phi}(k,f)\, \exists x,x' \in X\,  \forall m\in[n,f(n)]\, \left(A(k,m,x) \wedge B(k,m,x')\right),\]
	where $\overline{\Phi}(k,f):=\overline{\Phi}[\phi_1,\phi_2](k,f):=\max\{\theta,\phi_2(k,\widetilde{f}[\theta])\}$, with 
\begin{enumerate}
\item[] $\overline{f}(m):=\widetilde{f}[\phi_2(k,\widetilde{f}[m])](m)$
\item[] $\theta:=\phi_1(k,\overline{f})$  
\end{enumerate}
\end{proposition}
\begin{proof}
	Let $k\in\N$ and monotone $f:\N\to\N$ be given. By $(i)$ applied to $k$ and $\overline{f}$, there exist $n_1\leq \phi_1(k,\overline{f})$  and $x_1 \in X$ such that $\forall m\in [n_1,\overline{f}(n_1)]\, A(k,m,x_1)$.
By $(ii)$ applied to $k$ and $\widetilde{f}[n_1]$, there exist $n_2\leq \phi_2(k,\widetilde{f}[n_1])$ and $x_2 \in X$ satisfying $\forall m\in[n_2,\widetilde{f}[n_1](n_2)]  B(k,m,x_2)$.
	We will now check that $n:=\max\{n_1,n_2\}$ satisfies the desired conclusion.
	Observe that since $\phi_2$ is monotone, $n\leq \overline{\Phi}(k,f)$.
	By the definition of $n$ and the monotonicity of the functions $f$ and $\phi_2$, we have
	\[\begin{split}
		&[n,f(n)]\subseteq [n_1,f(\max\{n_1,\phi_2(k,\widetilde{f}[n_1])\})]=[n_1,\overline{f}(n_1)]\text{ and }\\
		&[n,f(n)]\subseteq [n_2,f(n)]=[n_2,\widetilde{f}[n_1](n_2)]\end{split}\] 
and the result follows.
\end{proof}
\begin{remark}\label{r:metameta}
The fact that one could change the order of $(i)$ and $(ii)$ in Proposition~\ref{p:metameta} entails that a better bound is the minimum of those two possibilities. Indeed, one can take $$\Phi[\phi_1,\phi_2](k,f):=\min\{\overline{\Phi}[\phi_1,\phi_2](k,f), \overline{\Phi}[\phi_2,\phi_1](k,f)\}.$$ 
\end{remark}
We finish this subsection with a general result that allows for discussions by cases in a even/odd distinction.
\begin{proposition}\label{p:metameta2}
	Let $X$ be a set, $A(k,m,x)$ be a formula with parameters $k,m,x$,  and $\psi$ be a monotone function satisfying
\[\forall k\in \N\, \mforall f \in \N \to \N \,\exists n\leq \psi(k,f)\, \exists x \in X\,  \forall m\in[n,f(n)]\, \left(A(k,2m,x) \wedge A(k,2m+1,x)\right).\]
Then \[\forall k\in \N\, \mforall f \in \N \to \N \,\exists n\leq \Psi(k,f)\, \exists x \in X\,  \forall m\in[n,f(n)]\, A(k,m,x),\] where 
$\Psi(k,f):=\Psi[\psi](k,f):=2\psi(k,g_f)+1$.
\end{proposition}
\begin{proof}
Let $k\in\N$ and monotone $f:\N\to\N$ be given. By the assumption there exist $n_0\leq \psi(k,g_f)$ and  $x_0 \in X$ such that  
\begin{equation}\label{e:parity}
\forall m\in[n_0,g_f(n_0)]\, \left(A(k,2m,x_0) \wedge A(k,2m+1,x_0)\right).
\end{equation}
By taking $n:=2n_0+1$, the result follows from \eqref{e:parity} considering the cases $m=2m'$ and $m=2m'+1$, for $m \in [n,f(n)]=[2n_0+1,g_f(n_0)]$, as in both cases $m' \in [n_0,g_f(n_0)]$. 
\end{proof}

\subsection{Quantitative lemmas}

In this section we present some useful technical lemmas. The following result is due to Boikanyo and Moro\c{s}ano, and generalizes a result by Xu \cite[Lemma~2.5]{X(02)}. 
\begin{lemma}[{\cite[Lemma~2.4]{BM(13)}}]\label{LemmaXuGen}
Let $(s_n)$ be a sequence of nonnegative real numbers satisfying $$s_{n+1}\leq (1-\alpha_n)(1-\lambda_n)s_n+ \alpha_nb_n + \lambda_n c_n+d_n,$$
where the sequences $(\alpha_n),(\lambda_n) \subset (0,1)$, $(b_n),(c_n)\subset \R$, and $(d_n) \subset \R^+_0 $ are such that:  $(i)$ $\sum \alpha_n = \infty$ $($or equivalently $\prod (1-\alpha_n)=0)$; $(ii)$ $\limsup b_n \leq 0$; $(iii)$ $\limsup c_n \leq 0$; and $(iv)$ $\sum d_n < \infty$.
Then $\lim s_n = 0$.
\end{lemma}
Adapting  \cite[Proposition~3.4]{LLPP(ta)} and  \cite[Lemma~14]{PP(ta)} we obtain Lemmas~\ref{l:Xuquant3} and \ref{l:Xuquant1} below which are quantitative versions of Lemma~\ref{LemmaXuGen}. 
\begin{lemma}\label{l:Xuquant3}
	Let $(s_n)$ be a bounded sequence of nonnegative real numbers and $M\in\N$ be a positive upper bound on $(s_n)$. Consider sequences of real numbers $(\alpha_n),(\lambda_n)\subset\, [0,1]$, $(b_n),(c_n)\subset \R$ and $(d_n)\subset \R^+_0$. Assume that $\sum \alpha_n = \infty$ (or $\sum \lambda_n= \infty$),  with rate of divergence ${\rm A}$, $\limsup b_n \leq 0$ and  $\limsup c_n \leq 0$, with rates ${\rm B}$ and ${\rm C}$ respectively, and $\sum d_n < \infty$ with Cauchy rate ${ \rm D}$. Assume that, for all $m \in \N$
	\[s_{m+1}\leq (1-\alpha_m)(1-\lambda_m)s_m+ \alpha_mb_m + \lambda_m c_m+d_m.\]
	Then $s_n \to 0$ with rate of convergence  
 $\rho_1(k):= \rho_1[{\rm A},{\rm B},{\rm C},{\rm D},M](k):={\rm A}\left(\widetilde{n}+\lceil \ln(4M(k+1))\rceil\right)+1$,
where $\widetilde{n}:= \max\{{\rm B}(4k+3), {\rm C}(4k+3),{\rm D}(4k+3)+1\}.$ 
\end{lemma}
\begin{lemma}\label{l:Xuquant1}
	Let $(s_n)$ be a bounded sequence of real numbers and $M\in\N$ a positive upper bound on $(s_n)$. Consider sequences of real numbers $(\alpha_n),(\lambda_n)\subset\, [0,1]$, $(v_n),(b_n),(c_n)\subset \R$ and assume the existence of a monotone function ${\rm A}$ which is a rate of divergence for $\sum \alpha_n= \infty$ (or $\sum \lambda_n = \infty$). For natural numbers $k, n$ and $p$ assume
	\[\forall m\in[n,p]\, \left(v_m\leq \frac{1}{4(k+1)(p+1)}\land b_m\leq \frac{1}{4(k+1)} \wedge c_m\leq \frac{1}{4(k+1)} \right),\]
	and 
	\[\forall m\in\N \,\left(s_{m+1}\leq (1-\alpha_m)(1-\lambda_m)(s_m+v_m)+ \alpha_mb_m + \lambda_m c_m\right).\]
	Then 	\[\forall m\in[\sigma_1(k,n),p]\, \left(s_m\leq \frac{1}{k+1}\right),\]
	with $\sigma_1(k,n):= \sigma_1[{\rm A},M](k,n):={\rm A}\left(n+\lceil \ln(4M(k+1))\rceil\right)+1$.
\end{lemma}
Consider the condition
\begin{equation}\label{c2'}
\forall m\in \N\, \left(\prod_{i= m}^{\infty} (1-\alpha_i)=0\right).
\end{equation}

One can equivalently work with this condition {\rm (\ref{c2'})} instead of the condition $\sum \alpha_n = \infty$. As such, it makes sense to also consider a quantitative hypothesis corresponding to {\rm (\ref{c2'})}:
\begin{equation}\label{q2'}
\begin{gathered}
{\rm A'}:\N\times \N\to\N \text{ is a monotone function satisfying}\\
\forall k, m\in \N\, \left( \prod_{i=m}^{{\rm A'}(m,k)}(1-\alpha_i)\leq \frac{1}{k+1}\right),
\end{gathered}
\end{equation}
implying that for each $m\in\N$, $A'(m, \cdot)$ is a rate of convergence towards zero for the sequence $\left(\prod_{i=m}^{n}(1-\alpha_i)\right)_n$. By saying that $\mathrm A'$ is monotone we mean that it is monotone in both variables,
\[
\forall k, k', m, m' \in \N\, \left( k\leq k' \land m\leq m' \to {\rm A'}(m, k)\leq {\rm A'}(m', k')\right).
\]

For some sequences $(\alpha_n)$, switching between these two conditions may prove to be useful since a rate of divergence for $(\sum \alpha_n)$ may have different complexity than a function satisfying {\rm (\ref{q2'})}. An easy example of this is the sequence $(\frac{1}{n+1})$ which has linear rates of convergence towards zero for $\left(\prod_{i= m}^{n} (1-\frac{1}{i+1})\right)_n$, but only an exponential rate of divergence for $\left(\sum_{i=0}^{n} \frac{1}{i+1}\right)_n$.

Next we state versions of Lemmas~\ref{l:Xuquant3} and \ref{l:Xuquant1} with a function $\mathrm A'$ satisfying condition \eqref{q2'} -- see \cite[Lemma~2.4]{K(15)}, \cite[Proposition~3.5]{LLPP(ta)}, and \cite[Lemma~16]{PP(ta)}.
\begin{lemma}\label{l:Xuquant4}
Let $(s_n)$ be a bounded sequence of nonnegative real numbers and $M\in\N$ be a positive upper bound on $(s_n)$. Consider sequences of real numbers $(\alpha_n),(\lambda_n)\subset\, [0,1]$, $(b_n),(c_n)\subset \R$ and $(d_n)\subset \R^+_0$. Assume that $\prod (1-\alpha_n)=0$ (or $\prod (1-\lambda_n)=0$), with ${\rm A'}$ a function satisfying the condition {\rm (\ref{q2'})},  $\limsup b_n \leq 0$ and  $\limsup c_n \leq 0$, with rates ${\rm B}$ and ${\rm C}$ respectively, and $\sum d_n < \infty$ with Cauchy rate ${ \rm D}$. Assume that, for all $m \in \N$
	\[s_{m+1}\leq (1-\alpha_m)(1-\lambda_m)s_m+ \alpha_mb_m + \lambda_m c_m+d_m.\]
	Then $s_n \to 0$ with rate of convergence
 $\rho_2(k):=\rho_2[{\rm A'}, {\rm B}, {\rm C}, {\rm D},M](k):={\rm A'}\left(\widetilde{n},4M(k+1)-1\right)+\widetilde{n}+1$,
where $\widetilde{n}:= \max\{{\rm B}(4k+3), {\rm C}(4k+3),{\rm D}(4k+3)+1\}$.
\end{lemma}
\begin{lemma}\label{l:Xuquant2}
Let $(s_n)$ be a bounded sequence of real numbers and $M\in\N$ a positive upper bound on $(s_n)$. Consider sequences of real numbers $(\alpha_n),(\lambda_n)\subset\, [0,1]$, $(v_n),(b_n),(c_n)\subset \R$ and assume the existence of a monotone function ${\rm A'}:\N \times \N \to \N$ satisfying condition {\rm (\ref{q2'})}. For natural numbers $k, n$ and $p$ assume
	\[\forall m\in[n,p]\, \left(v_m\leq \frac{1}{4(k+1)(p+1)}\land b_m\leq \frac{1}{4(k+1)} \wedge c_m\leq \frac{1}{4(k+1)} \right),\]
	and 
	\[\forall m\in\N \,\left(s_{m+1}\leq (1-\alpha_m)(1-\lambda_m)(s_m+v_m)+ \alpha_mb_m + \lambda_m c_m\right).\]
	Then
	\[\forall m\in[\sigma_2(k,n),p]\, \left(s_m\leq \frac{1}{k+1}\right),\]
	with $\sigma_2(k,n):=\sigma_2[{\rm A'}, M](k,n):={\rm A'}\left(n,4M(k+1)-1\right)+1$.
\end{lemma}
%
%
The next result is due to Suzuki. A (partial) quantitative version of Lemma~\ref{SuzukiLemma} was obtained in \cite{DP(20)} through an arithmetization of a certain $\limsup$. 
\begin{lemma}\label{SuzukiLemma}\textup{(\cite[Lemma 2.2 ]{Suzuki2005})}
Let $(z_n)$ and $(w_n)$ be bounded sequences in a Banach space $X$ and let $(\alpha_n)$ be a sequence in $[0,1]$ with $0< \liminf \alpha_n \leq \limsup \alpha_n <1$. Suppose that $z_{n+1}=\alpha_{n}w_n+(1-\alpha_n)z_n$ for all $n \in \N$, and $\limsup (\norm{w_{n+1} -w_n}-\norm{z_{n+1}-z_n})\leq 0$. Then $\lim \norm{w_n-z_n}=0$.
\end{lemma}
\begin{lemma}\label{LemmaSuzuki2}
Let $(z_n),(w_n)$ be sequences in a normed space $X$ and $N \in \N$ be such that $\norm{z_n},\norm{w_n}\leq N$, for all $n \in \N$. Let $(\alpha_n) \subset [0,1]$ be a sequence of real numbers and $a \in \N \setminus \{0\}$ be such that $\forall n \geq a \left(\frac{1}{a}\leq\alpha_n \leq 1 - \frac{1}{a}\right).$
Suppose that $z_{n+1}=\alpha_nw_n+(1-\alpha_n)z_n$, for all $n \in \N$ and that there exists a monotone function $\nu: \N \to \N$ such that 
\begin{equation}\label{eqnu}
\forall k \in \N \,\forall n \geq \nu(k) \left(\norm{w_{n+1}-w_n}-\norm{z_{n+1}-z_n}\leq \frac{1}{k+1}\right).
\end{equation}
Then $\chi$ is a quasi-rate of convergence for $\norm{w_m - z_m} \to 0$,
where $\chi=\chi[a,\nu,N]$, is the monotone function $\widetilde{\chi}$ from $\cite[Lemma~4.9]{DP(20)}$.
\end{lemma}
The next two results give a quantitative strengthening of the fact that if $b \in \A(a)$ and $x$ is a zero of $\A$, then $\ip{ a-x,b} \geq 0.$
\begin{lemma}\label{l:monotone1}
Let $\A$ be a monotone operator and $\lambda >0$ be a real number. For $a, b, x \in H$ and $k \in \N$ we have 
\begin{equation*}
b \in \A(a) \to \ip{ a-x,b} \geq - \norm{J_{\lambda}^{\A}(x)- x}\left(\frac{\norm{J_{\lambda}^{\A}(x)-a}}{\lambda}+\norm{b} \right).
\end{equation*}
\end{lemma}
\begin{proof}
 By the definition of the resolvent function $x \in J_{\lambda}^{\A}(x)+ \lambda\A(J_{\lambda}^{\A}(x))$. Then $w \in \A(J_{\lambda}^{\A}(x))$, with $w= \frac{1}{\lambda}(x-J_{\lambda}^{\A}(x))$. Hence, using the monotonicity of $\A$, 

\begin{equation*}
\begin{split}
\ip{ a-x,b} &= \ip{ a-J_{\lambda}^{\A}(x),b } + \ip{ J_{\lambda}^{\A}(x)-x,b } \geq \ip{ a- J_{\lambda}^{\A}(x),w} +  \ip{ J_{\lambda}^{\A}(x)-x,b }\\
& \geq -\norm{J_{\lambda}^{\A}(x)-a} \norm{w} -\norm{J_{\lambda}^{\A}(x)-x}\norm{b} = - \norm{J_{\lambda}^{\A}(x)- x}\left(\frac{\norm{J_{\lambda}^{\A}(x)-a}}{\lambda}+\norm{b} \right). \qedhere
\end{split}
\end{equation*} 
\end{proof}
As a direct consequence of Lemma~\ref{l:monotone1} we obtain following lemma.
\begin{lemma}\label{l:monotone0}
Let $\A$ be a monotone operator and $\lambda >0$ be a real number. For $a, b, x \in H$ and $k \in \N$ we have 
\begin{equation*}
\left(b \in \A(a) \wedge \norm{J_{\lambda}^{\A}(x)-x}\leq \frac{\lambda}{2M(k+1)} \right)\to \ip{ a-x,b} \geq -\frac{1}{k+1},
\end{equation*}
where $M \in \N \setminus \{0\}$ is such that $M \geq \max\{\norm{J_{\lambda}^{\A}(x)-a}\!, \lambda\norm{b}\}$.
\end{lemma}
\begin{notation}
	For $q \in S$  and $N \in \N$, we denote by $B_{N}$ the closed ball centred at $q$ with radius $N$, i.e.\ $
	B_{N}:=\{ z \in H: \norm{z-q}\leq N\}.$ In the following, a point $q$ is always made clear from the context.
\end{notation}
The following lemma is an easy adaptation of \cite[Proposition~7]{PP(ta)}.
\begin{lemma}\label{l:projection}
	Let $\A:D(\A) \subset H \to 2^H$ and $\B:D(\B) \subset H \to 2^H$ be maximal monotone operators on a Hilbert space $H$. Assume that $S:= \A^{-1}(0) \cap \B^{-1}(0) \neq \emptyset$ and consider the resolvent functions $J^\A:=(Id+\frac{1}{R}\A)^{-1}$ and $J^\B:=(Id+\frac{1}{R}\B)^{-1}$, where $R \in \N \setminus\{0\}$. Take $u\in H$ and $N\in \N\setminus \{0\}$ a natural number satisfying $N\geq 2\|u-q\|$ for some point $q\in S$.
	For any $k\in \N$ and monotone function $f:\N \to \N $, there are $n \leq \zeta(k,f)$ and $x\in B_N$ such that
\begin{equation*}
\|J^\A(x)-x \| \leq \frac{1}{f(n)+1} \, \wedge \|J^\B(x)-x \| \leq \frac{1}{f(n)+1}
\end{equation*}
and 
\begin{equation*}
\forall y\in B_N  \left(\left(\|J^\A(y)-y\|\leq \frac{1}{n+1} \wedge \|J^\B(y)-y\|\leq \frac{1}{n+1}\right)\to \ip{ u-x,y-x} \leq \frac{1}{k+1}\right),
\end{equation*} 
	with $\zeta(k,f):=\zeta[N](k,f):=24N(w_{f,N}^{(E)}(0)+1)^2$, where $w_{f,N}:=\max\{ f(24N(m+1)^2),\, 24N(m+1)^2 \}$ and $E:=E[N,k]:=4N^4(k+1)^2$.
\end{lemma}
As usual, the first step to prove strong convergence is to show that the sequence given by the algorithm is bounded. An easy induction argument gives upper bounds on a sequence $(y_n)$ generated by \eqref{HPPA2}.
\begin{lemma}\label{l:vbounded}
Let $N \in \N$ be such that $N \geq \max\{\norm{u-q}, \norm{x_0-q}, \norm{q}\}$, for some $q \in S$. Then $\norm{y_n -q} \leq N$, for all $n \in \N$. In particular, $(y_n)$ is bounded with $\norm{y_n} \leq 2N$, for all $n \in \N$.
\end{lemma}
%
\section{Metastability for \texorpdfstring{$\eqref{HPPA2}$}{MAR*}}\label{s:3}

In the following we will assume that $\A:D(\A) \subset H \to 2^H$ and $\B:D(\B) \subset H \to 2^H$ are maximal monotone operators such that $S:= \A^{-1}(0) \cap \B^{-1}(0) \neq \emptyset$. For arbitrary but fixed vectors $x_0, u \in H $, we will denote by $(x_n)$ a sequence generated by \eqref{HPPA2e} and by $(y_n)$ the corresponding ``error-free'' sequence generated by \eqref{HPPA2} (with $y_0=x_0$).
We assume that there exist $R \in \N\setminus \{0\}$ and monotone functions $a,\ell,{\rm A}, r_{\beta}, r_{\mu},t: \N \to \N$ such that 
\begin{enumerate}[($Q_1$)]
\item $a$ is a rate of convergence for $\alpha_n \to 0$ 
\item $\ell$ is a rate of convergence for $\lambda_n \to 0$  
\item $A$ is a rate of divergence for $(\sum\alpha_n)$ or $(\sum\lambda_n)$ 
\item $\forall n \in \N\left(\min\{\beta_n, \mu_n\} \geq \frac{1}{R}\right)$
\item $\forall n \in \N\left(\max\{\beta_n, \mu_n\}\leq t(n) \right)$
\item $r_{\beta}$ is a rate of convergence for $\frac{\beta_{n+1}}{\beta_n} \to 1 $
\item $r_{\mu}$ is a rate of convergence for $\frac{\mu_{n+1}}{\mu_n} \to 1 $.
\end{enumerate} 
We write $J^\A:=J^\A_{R^{-1}}$ and $J^\B:=J^\B_{R^{-1}}$.

In Subsection~\ref{s:quantitativeBM} we obtain an effective partial metastability bound on the iteration \eqref{HPPA2} (cf. Corollary~\ref{c:quantBM} below). This metastability property is partial in the sense that it is obtained under the conditions $(Q_3)-(Q_5)$ together with:
\begin{enumerate}
\item[($Q_{\eta}$)$\phantom{'}$] $\eta$ is a monotone quasi-rate of asymptotic regularity for $(y_n)$ w.r.t.\ $J^{\A}$ and $J^{\B}$
\item[($Q_{\eta'}$)] $\eta'$ is a monotone quasi-rate of convergence for $\norm{y_{n+1}-y_n} \to 0$.
\end{enumerate}
 Subsection~\ref{s:qrar} shows that it is possible to satisfy conditions ($Q_{\eta}$) and ($Q_{\eta'}$), and indeed obtain an effective metastability bound, under the assumptions $(Q_1),(Q_2),(Q_4),(Q_6)$ and $(Q_7)$ -- cf. Remarks~\ref{r:eta'} and \ref{r:eta}.
\subsection{Conditional metastability}\label{s:quantitativeBM}
%
The next quantitative lemma replaces the original sequential weak compactness argument.
\begin{lemma}\label{l:removal}
	Let $N\in \N\setminus \{0\}$ be such that $N\geq \max\{2\norm{u-q}, \norm{x_0-q}, \norm{q}\}$ for some point $q\in S$.
	For any $k\in \N$ and monotone function $f:\N \to \N $, there are $n \leq \Omega_{\eta}(k,f)$ and $x\in B_N$ such that
\begin{equation*}
\|J^\A(x)-x \| \leq \frac{1}{f(n)+1} \, \wedge \|J^\B(x)-x \| \leq \frac{1}{f(n)+1}
\end{equation*}
and
\begin{equation*}
\forall m\in [n,f(n)] \left(\ip{ u-x,y_m-x} \leq \frac{1}{k+1}\right),
\end{equation*} 
where $\Omega_{\eta}(k,f):=\Omega_{\eta}[N](k,f):=\eta(\zeta(k,\widehat{f}),f)$, with $\zeta=\zeta[N]$ as in Lemma~\ref{l:projection}, and  $\widehat{f}(m):=f(\eta(m,f))$.
\end{lemma}
\begin{proof}
By Lemma~\ref{l:projection}, there exist $n_0 \leq \zeta(k,\widehat{f})$ and $x \in B_N$ such that
\begin{equation*}
\|J^\A(x)-x \| \leq \frac{1}{\widehat{f}(n_0)+1} \, \wedge \|J^\B(x)-x \| \leq \frac{1}{\widehat{f}(n_0)+1}
\end{equation*}
and 
\begin{equation*}
\forall y\in B_N  \left(\left(\|J^\A(y)-y\|\leq \frac{1}{n_0+1} \wedge \|J^\B(y)-y\|\leq \frac{1}{n_0+1}\right)\to \ip{ u-x,y-x} \leq \frac{1}{k+1}\right).
\end{equation*} 
By ($Q_{\eta}$), there exists $n \leq \eta(n_0,f)$ such that
\begin{equation*}
 \forall m \in [n,f(n)] \left(\norm{J^{A}(y_{m})-y_{m}}\leq \frac{1}{n_0+1} \wedge  \norm{J^{B}(y_{m})-y_{m}} \leq \frac{1}{n_0+1} \right).
\end{equation*}
Since $f$ is monotone we have $\widehat{f}(n_0)\geq f(n)$. The result then follows from the fact that $(y_n)\subset B_N$.
\end{proof}
Next we show the main result of this section, which allows to obtain a rate of metastability for $(y_n)$.
\begin{theorem}\label{t:quantBM}
For $x_0, u \in H $, let $(y_n)$ be generated by \eqref{HPPA2}. Consider $N\in \N\setminus \{0\}$ such that $N\geq \max\{2\norm{u-q}, \norm{x_0-q}, \norm{q}\}$ for some point $q\in S$. Assume the condition $(Q_{\eta})$ and let $\Omega_{\eta}$ be as in Lemma~\ref{l:removal}. Assume that there exist $R \in \N\setminus \{0\}$ and monotone functions ${\rm A},t: \N \to \N$ satisfying conditions $(Q_3) - (Q_5)$.
For all $k \in \N$ and monotone function $f:\N\to \N$
\begin{equation*}
\exists n \leq \widetilde{\mu}(k,f) \, \exists x \in B_N \, \forall m\in [n,f(n)] \left(\norm{y_{2m}-x}\leq \frac{1}{k+1} \wedge \|J^\A(x)-x \| \leq \frac{1}{k+1} \wedge \|J^\B(x)-x \| \leq \frac{1}{k+1}\right),
\end{equation*}
where $\widetilde{\mu}(k,f):=\widetilde{\mu}[\eta,N,R,{\rm A},t](k,f):=\sigma_1(\widetilde{k},\Omega_{\eta}(k',2g+2))$, with $\sigma_1:=\sigma_1[A,4N^2]$ as in Lemma~\ref{l:Xuquant1} and
\begin{enumerate}
\item[] $g(n):= 64NR(k+1)^2(p(n)+1)t(p(n))-1$ 
\item[] $k':=16(k+1)^2-1$ 
\item[] $p(n):=f(\sigma_1(\widetilde{k},n))$ 
\item[] $\widetilde{k}:=(k+1)^2-1$.
\end{enumerate}
\end{theorem}
\begin{proof}
We may assume that for all $m \leq \widetilde{\mu}(k,f)$ it holds that $f(m) \geq m$, otherwise the result is trivial.
By Lemma~\ref{l:removal} there exist $n_0 \leq \Omega_{\eta}(k',2g+2)$ and $x_0\in B_N$ such that 
\begin{equation}\label{e:almostzero}
\|J^\A(x_0)-x_0 \| \leq \frac{1}{g(n_0)+1} \, \wedge \|J^\B(x_0)-x_0 \| \leq \frac{1}{g(n_0)+1}
\end{equation}
and
\begin{equation}\label{e:prodintsmalli}
\forall m\in [n_0,2p+2] \left(\ip{ u-x_0,y_m-x_0} \leq \frac{1}{16(k+1)^2}\right),
\end{equation} 
with $p:=p(n_0)$.
From \eqref{e:prodintsmalli} we derive that for all $m\in [n_0,p]$
\begin{equation}\label{e:prodintsmall1}
\ip{ u-x_0,y_{2m+1}-x_0} \leq \frac{1}{16(k+1)^2} \mbox{ \, and \,} \ip{ u-x_0,y_{2m+2}-x_0} \leq \frac{1}{16(k+1)^2}.
\end{equation} 
For all $m \in \N$, by the definition of $y_{2m+2}$ we have
\begin{equation*}
\B(y_{2m+2}) \ni \frac{1}{\mu_m}\left(\lambda_m(u-x_0)+(1-\lambda_m)(y_{2m+1}-x_0)-(y_{2m+2}-x_0) \right).
\end{equation*}
We have that
\begin{equation*}
\max \left\{\norm{J^{\B}(x_0)-y_{2m+2}},\frac{1}{R}\norm{\frac{1}{\mu_m}\left(\lambda_mu+(1-\lambda_m)y_{2m+1}-y_{2m+2} \right)}\right\} \leq 2N.
\end{equation*}
Hence, by Lemma~\ref{l:monotone0} and \eqref{e:almostzero}
\begin{equation*}
\ip{ y_{2m+2}-x_0,\frac{1}{\mu_m}\left(\lambda_m(u-x_0)+(1-\lambda_m)(y_{2m+1}-x_0)-(y_{2m+2}-x_0) \right)} \geq -\frac{1}{16(k+1)^2(p+1)t(p)}
\end{equation*}
Hence
\begin{equation*}
\begin{split}
2 \norm{y_{2m+2}-x_0}^2 &\leq 2(1-\lambda_m)\ip{ y_{2m+2}-x_0, y_{2m+1}-x_0} \\
& \qquad +2 \lambda_m \ip{ y_{2m+2}-x_0,u-x_0} + \frac{\mu_m}{8(k+1)^2(p+1)t(p)}\\
& \leq (1-\lambda_m)\left(\norm{y_{2m+2}-x_0}^2 +\norm{y_{2m+1}-x_0}^2\right) +2 \lambda_m \ip{ y_{2m+2}-x_0,u-x_0}\\
& \qquad + \frac{\mu_m}{8(k+1)^2(p+1)t(p)},
\end{split}
\end{equation*}
which implies that 
\begin{equation}\label{e:y2n+2}
 \norm{y_{2m+2}-x_0}^2 \leq (1-\lambda_m)\norm{y_{2m+1}-x_0}^2 +2 \lambda_m \ip{ y_{2m+2}-x_0,u-x_0} + \frac{\mu_m}{8(k+1)^2(p+1)t(p)}.
\end{equation}
Similarly,
\begin{equation}\label{e:y2n+1}
 \norm{y_{2m+1}-x_0}^2 \leq (1-\alpha_m)\norm{y_{2m}-x_0}^2 +2 \alpha_m \ip{ y_{2m+1}-x_0,u-x_0} + \frac{\beta_m}{8(k+1)^2(p+1)t(p)}.
\end{equation}
Combining \eqref{e:y2n+2} and \eqref{e:y2n+1} we derive
\begin{equation}\label{e:maintheorem}
 \norm{y_{2m+2}-x_0}^2 \leq (1-\alpha_m)(1-\lambda_m)(\norm{y_{2m}-x_0}^2 +v_m)+ \alpha_m b_m+ \lambda_m c_m,
\end{equation}
where
\begin{enumerate}
\item[] $v_m:= \dfrac{\mu_m+\beta_m}{8(k+1)^2(p+1)t(p)}$, 
\item[] $b_m:=2\left(\ip{ y_{2m+1}-x_0,u-x_0} +\dfrac{\mu_m+\beta_m}{16(k+1)^2(p+1)t(p)}\right)$, and 
\item[] $c_m:=2\left(\ip{ y_{2m+2}-x_0,u-x_0} + \dfrac{\mu_m}{16(k+1)^2(p+1)t(p)}\right)$.
\end{enumerate}
For $m \in [n_0,p]$, we have 
\begin{equation*}
v_m \leq  \frac{4t(p)}{16(k+1)^2(p+1)t(p)}\leq \frac{1}{4(k+1)^2(p+1)} 
\end{equation*}
and
\begin{equation*}
c_m \leq  2\left(\ip{ y_{2m+2}-x_0,u-x_0} + \frac{t(p)}{16(k+1)^2t(p)}\right)\leq \frac{1}{4(k+1)^2} .
\end{equation*}
Notice that $\sigma_1(\widetilde{k},n_0) \leq \widetilde{\mu}(k,f)$, using the monotonicity of $\sigma_1$. By the assumption on $f$, $p \geq \sigma_1(\widetilde{k},n_0) \geq 1$ which entails, for $m \in [n_0,p]$
\begin{equation*}
b_m \leq  2\left(\ip{ y_{2m+1}-x_0,u-x_0} +\frac{2t(p)}{32(k+1)^2t(p)}\right)\leq \frac{1}{4(k+1)^2} .
\end{equation*}
By Lemma~\ref{l:Xuquant1} we conclude that for $m \in [\sigma_1(\widetilde{k},n_0),f(\sigma_1(\widetilde{k},n_0))]$ 
\begin{equation*}
\norm{y_{2m}-x_0}^2\leq\frac{1}{(k+1)^2},
\end{equation*}
and the result follows with $n:=\sigma_1(\widetilde{k},n_0)$ and $x=x_0$.
\end{proof}
\begin{corollary}\label{c:quantBM}
Under the conditions of Theorem~\ref{t:quantBM} and $(Q_{\eta'})$, we have that for all $k \in \N$ and monotone function $f:\N\to \N$
\begin{equation*}
\exists n \leq \mu(k,f) \, \forall i,j\in [n,f(n)] \left(\norm{y_{i}-y_j}\leq \frac{1}{k+1} \right),
\end{equation*}
where $\mu(k,f):=\Psi[\Phi[\psi_1,\psi_2]](k,f)$, with $\psi_1(k,f):=\eta'(4k+3,2f)$ and $\psi_2(k,f):=\widetilde{\mu}(4k+3,f)$, $\Psi$ as in Proposition~\ref{p:metameta2}, $\Phi$ as in Remark~\ref{r:metameta}, and $\widetilde{\mu}$ is as in Theorem~\ref{t:quantBM}.
\end{corollary}
\begin{proof}
From $(Q_{\eta'})$ we obtain
\begin{equation*}
\forall k \in \N \, \mforall f: \N \to \N \, \exists n \leq \psi_1(k,f) \,\forall m \in [n,f(n)] \left(\norm{y_{2m+1}-y_{2m}}\leq \frac{1}{4(k+1)} \right),
\end{equation*}
and by Theorem~\ref{t:quantBM}
\begin{equation*}
\forall k \in \N \, \mforall f: \N \to \N \,\exists n \leq \psi_2(k,f)\, \exists x \in B_N \,\forall m \in [n,f(n)] \left(\norm{y_{2m}-x}\leq \frac{1}{4(k+1)} \right).
\end{equation*}
By Remark~\ref{r:metameta}, for all  $k \in \N$ and monotone  $f: \N \to \N$  there exist $ n \leq \Phi[\psi_1,\psi_2](k,f)$ and  $x \in B_N$ such that
\begin{equation*}
\forall m \in [n,f(n)] \left(\norm{y_{2m+1}-y_{2m}}\leq \frac{1}{4(k+1)} \wedge \norm{y_{2m}-x}\leq \frac{1}{4(k+1)} \right).
\end{equation*}
For $m \in [n,f(n)]$,
\begin{equation*}
\norm{y_{2m+1} -x} \leq \norm{y_{2m+1}-y_{2m}}+\norm{y_{2m}-x} \leq \frac{1}{2(k+1)}.
\end{equation*}
Hence, by Proposition~\ref{p:metameta2}, there exists $ n\leq \mu(k,f)$ such that for $i,j \in [n,f(n)]$
\begin{equation*}
\norm{y_{i} -y_{j}} \leq \norm{y_{i}-x}+\norm{y_{j}-x} \leq \frac{1}{2(k+1)}+\frac{1}{2(k+1)}=\frac{1}{k+1},
\end{equation*}
which entails the result.
\end{proof}
\subsection{Asymptotic regularity and metastability}\label{s:qrar}
We now show that under the assumptions $(Q_1),(Q_2),(Q_4),(Q_6)$ and $(Q_7)$ it is possible to satisfy the conditions ($Q_{\eta}$) and ($Q_{\eta'}$). This implies that the results of the previous subsection hold under the conditions $(Q_1)-(Q_7)$.
\begin{lemma}\label{l:asymreg1}
Let $N \in \N$ be such that $N \geq \max\{\norm{u-q}, \norm{x_0-q}, \norm{q}\}$, for some $q \in S$. Assume that there exist monotone functions $a,\ell, r_{\beta}, r_{\mu}: \N \to \N$ satisfying conditions $(Q_1), (Q_2), (Q_6)$ and $(Q_7)$.
Then
\begin{enumerate}[$(i)$]
\item\label{qrcodd} $\norm{y_{2n+1}-y_{2n-1}} \to 0$ with monotone quasi-rate of convergence $\eta_{0}(k,f):=\chi[4,\nu,3N](k,f)$
\item\label{qrceven} $\norm{y_{2n+2}-y_{2n}} \to 0$ with monotone quasi-rate of convergence $\eta_{1}(k,f):=\max\{\eta_0(2k+1,\widetilde{f}[N_1]),N_1 \}$
\end{enumerate}
where $\chi$ is the monotone function from Lemma~\ref{LemmaSuzuki2} and
\begin{enumerate}
\item[] $\nu(k):= \max \{a(38N(k+1)-1), \ell(22N(k+1)-1)+1,r_{\beta}(11N(k+1)-1), r_{\mu}(16N(k+1)-1)+1 \},$
\item[] $N_1:= \max\{\ell(16N(k+1)-1)+1,r_{\mu}(16N(k+1)-1)+1\}$
\end{enumerate}
\end{lemma}
\begin{proof}
By the resolvent identity we have, for some $R \in \N\setminus \{0\}$
\begin{equation}
y_{2n+2}=J^{\B}\left(\frac{1}{R\mu_n}(\lambda_n u+ (1-\lambda_n)y_{2n+1})+ \left(1-\frac{1}{R\mu_n} \right) y_{2n+2}\right).
\end{equation}
By Lemma~\ref{l:vbounded}, we have that $\norm{u-y_n} \leq 2N$ and $\norm{y_{n+1}-y_n} \leq 2N$, for all $n \in \N$.
Then, using the fact that $J^{\B}$ is nonexpansive we derive 
\begin{equation*}
\begin{split}
\norm{y_{2n+2}-y_{2n}}& \leq \left \|\frac{1}{R\mu_n}(\lambda_n u+ (1-\lambda_n)y_{2n+1})+ \left(1-\frac{1}{R\mu_n} \right) y_{2n+2}\right.\\
    & \quad \left. -\frac{1}{R\mu_{n-1}}(\lambda_{n-1} u+ (1-\lambda_{n-1})y_{2n-1})- \left(1-\frac{1}{R\mu_{n-1}} \right) y_{2n}) \right\|\\
& = \left \| \left(\frac{\lambda_n}{R\mu_n}-\frac{\lambda_{n-1}}{R\mu_{n-1}}\right) u + \frac{1-\lambda_n}{R\mu_n}(y_{2n+1}-y_{2n-1}) + y_{2n-1}\left(\frac{1-\lambda_n}{R\mu_n}-\frac{1-\lambda_{n-1}}{R\mu_{n-1}}\right) \right.\\
& \quad \left.+ \left(1-\frac{1}{R\mu_n} \right) y_{2n+2}- \left(1-\frac{1}{R\mu_{n-1}} \right) y_{2n} \right\|\\
& = \left \| \left(\frac{\lambda_n}{R\mu_n}-\frac{\lambda_{n-1}}{R\mu_{n-1}}\right) (u-y_{2n-1}) + \frac{1-\lambda_n}{R\mu_n}(y_{2n+1}-y_{2n-1}) \right.\\
& \quad \left. + \left(\frac{1}{R\mu_n}-\frac{1}{R\mu_{n-1}}\right)(y_{2n-1}-y_{2n}) + \left(1-\frac{1}{R\mu_n} \right) (y_{2n+2}-y_{2n})  \right\|\\
&   \leq \frac{1-\lambda_n}{R\mu_n}\norm{y_{2n+1}-y_{2n-1}} + \left(1-\frac{1}{R\mu_n} \right) \norm{y_{2n+2}-y_{2n}}\\
& \quad + 2N\left(\left | \frac{ \lambda_{n}}{R\mu_{n}}-\frac{ \lambda_{n-1}}{R\mu_{n-1}} \right | + \left | \frac{1}					{R\mu_{n}} - \frac{1}{R\mu_{n-1}} \right | \right).
\end{split}
\end{equation*}
Hence,
\begin{equation}\label{e:dif_pares}
\norm{y_{2n+2}-y_{2n}}  \leq (1-\lambda_{n})\norm{y_{2n+1}-y_{2n-1}} +2N\left( \lambda_{n}+\lambda_{n-1}+ (\lambda_{n-1} +1)  \left | 1 - \frac{\mu_{n}}{\mu_{n-1}} \right |\right) .
\end{equation}
Define $T_{n}^{\A}:=2J_{\beta_n}^{\A}-\Id$. For all $n \in \N$, since the resolvent function is firmly nonexpansive,  by Lemma~\ref{l:fnin},  $T_{n}^{\A}$ is nonexpansive.
By the definition of \eqref{HPPA2} we have $y_{2n+1}= \frac{z_n + T_{n}^{\A} z_n}{2}$, where $z_n:= \alpha_n u +(1-\alpha_n) y_{2n}$. Using the resolvent identity it holds
\begin{equation*}
\begin{split}
\norm{T_{n+1}^{\A} z_{n+1} - T_{n}^{\A} z_n}& \leq  \norm{T_{n+1}^{\A} z_{n+1} - T_{n+1}^{\A} z_n} + 					\norm{T_{n+1}^{\A} z_{n} - T_{n}^{\A} z_n}\\
   & \leq \norm{z_{n+1} - z_n} + 2 \norm{J_{\beta_{n+1}}^{\A} z_n - J_{\beta_{n}}^{\A} z_n}\\
   & = \norm{z_{n+1} - z_n} + 2\norm{J_{\beta_{n+1}}^{\A} z_n - J_{\beta_{n+1}}^{\A} z 		\left( \frac{\beta_{n+1}}{\beta_n} z_n + \left(1-\frac{\beta_{n+1}}{\beta_n}  \right)J_{\beta_{n}}^{\A} z_n \right)}\\
& \leq 2N(\alpha_{n+1}+\alpha_n) + \norm{y_{2n+2}-y_{2n}} + 2 \left | 1-\frac{\beta_{n+1}}{\beta_n} \right| \norm{z_n - J_{\beta_{n+1}}^{\A} z_n}.
\end{split}
\end{equation*}
Similarly we have $y_{2n+2}= \frac{w_n + F_{n}^{\B} w_n}{2}$, where $F_{n}^{\B}$ is the nonexpansive function defined by $F_{n}^{\B}:=2J_{\mu_n}^{\B}-\Id$, and where $w_n:= \lambda_n u +(1-\lambda_n) y_{2n+1}$. 
As above we have
\begin{equation*}
\begin{split}
\norm{F_{n}^{\B} w_{n} - F_{n-1}^{\B} w_{n-1}}& \leq \norm{w_{n} - w_{n-1}} + 2 \norm{J_{\mu_{n}}^{\B} w_{n-1} - 				J_{\mu_{n-1}}^{\B} w_{n-1}}\\
   & \leq 2N(\lambda_{n} + \lambda_{n-1}) +  \norm{y_{2n+1} - y_{2n-1}} + 2 \left | 1-\frac{\mu_{n}}{\mu_{n-1}} \right| 			\norm{w_{n-1} - J_{\mu_{n}}^{\B} w_{n-1}}.
\end{split}
\end{equation*}
The definitions above imply that 
\begin{equation}\label{e:34}
y_{2n+1}=\frac{1}{4}y_{2n-1}+\frac{3}{4} v_n,
\end{equation}
where $(v_n)$ is the bounded sequence defined by
\begin{equation*}
v_n :=\frac{\lambda_{n-1} (1-\alpha_n) (u-y_{2n-1}) + \alpha_n (2u -y_{2n-1}-F_{n-1}^{\B} w_{n-1})+(F_{n-1}^{\B} w_{n-1} + 2 T_{n}^{\A} z_n)}{3}.
\end{equation*}
Since $\norm{u-y_n} \leq 2N$ and $\norm{y_n}\leq 2N$, for all $n \in \N$ we have 
\begin{equation*}
\begin{split}
\norm{2u-y_{2n+1}-F_{n}^{\B} w_{n}}&=\norm{2u-y_{2n+1}-2y_{2n+2}-\lambda_{n} u -(1-\lambda_{n}) y_{2n+1}}\\  
& \leq 2 \norm{u-y_{2n+2}}+\norm{y_{2n+1}+\lambda_{n} u + (1-\lambda_{n}) y_{2n+1}}\\
& \leq 2 \norm{u-y_{2n+2}}+\norm{u-y_{2n+1}}+2 \norm{y_{2n+1}}\leq 10N.
\end{split}
\end{equation*}
Hence
\begin{equation*}
\begin{split}
3 \norm{v_{n+1}-v_n} & \leq 2N(\lambda_{n}+\lambda_{n-1})+ 10N(\alpha_{n+1}+\alpha_n)  + \norm{F_{n}^{\B} 				w_{n} - F_{n-1}^{\B} w_{n-1}} + 2 \norm{T_{n+1}^{\A} z_{n+1} - T_{n}^{\A} z_n}\\
& \leq 4 N (\lambda_{n}+\lambda_{n-1}) +  14N (\alpha_{n+1}+\alpha_n) +  3\norm{y_{2n+1} - y_{2n-1}} \\ 
	&\quad+ 2 \left | 1-\frac{\mu_{n}}{\mu_{n-1}} \right| \norm{w_{n-1} - J_{\mu_{n}}^{\B} w_{n-1}}+4 \left | 1-\frac{\beta_{n+1}}{\beta_n} \right| \norm{z_n - J_{\beta_{n+1}}^{\A} z_n}\\
	&\quad+4N\left( \lambda_{n}+\lambda_{n-1}+ (\lambda_{n-1} +1)  \left | 1 - \frac{\mu_{n}}{\mu_{n-1}} \right |\right).
\end{split}
\end{equation*}
Using Lemma~\ref{l:vbounded} we have 
\begin{equation*}
\begin{split}
\norm{z_n - J_{\beta_{n+1}}^{\A} z_n} & \leq \norm{z_n - q}+ \norm{J_{\beta_{n+1}}^{\A} z_n -q} \leq 2 \norm{z_n -q} \\
& \quad = 2\norm{\alpha_n (u-q) +(1-\alpha_n) (y_{2n} -q)} \leq 2N
\end{split}
\end{equation*}
and, similarly,  $ \norm{w_{n-1} - J_{\mu_{n}}^{\B} w_{n-1}}\leq 2N$. Since $\lambda_{n-1} \leq 1$ we then have 
\begin{equation*}
\begin{split}
\norm{v_{n+1}-v_n}- \norm{y_{2n+1} - y_{2n-1}} & \leq \frac{8}{3} N (\lambda_{n}+\lambda_{n-1}) +  \frac{14}{3}N (\alpha_{n+1}+\alpha_n)  \\ 
	&\quad+ 4N \left | 1-\frac{\mu_{n}}{\mu_{n-1}} \right|+\frac{8}{3}N \left | 1-\frac{\beta_{n+1}}{\beta_n} \right|.
\end{split}
\end{equation*}
The latter inequality entails that 
\begin{equation*}
	\forall k \in \N \, \forall n \geq \nu(k) \left(\norm{v_{n+1}-v_n} - \norm{y_{2n+1} - y_{2n-1}}\leq \frac{1}{k+1}\right),
\end{equation*}
where $\nu(k):= \max \{a(38N(k+1)-1), \ell(22N(k+1)-1)+1,r_{\beta}(11N(k+1)-1), r_{\mu}(16N(k+1)-1)+1 \}$. 
Indeed, by condition~($Q_2$) for $n \geq  \ell(22N(k+1)-1)+1$ 
\begin{equation*}
\frac{8}{3} N (\lambda_{n}+\lambda_{n-1}) \leq \frac{16N}{3(22N(k+1))} \leq \frac{1}{4(k+1)}, 
\end{equation*}
by condition~($Q_1$) for $n \geq a(38N(k+1)-1)$
\begin{equation*}
\frac{14}{3} N (\alpha_{n+1}+\alpha_{n}) \leq \frac{28N}{3(38N(k+1))} \leq \frac{1}{4(k+1)},  
\end{equation*}
by condition~($Q_7$) for $n \geq r_{\mu}(16N(k+1)-1)+1$
\begin{equation*}
4 \left| 1-\frac{\mu_{n}}{\mu_{n-1}} \right| \leq \frac{4N}{16N(k+1)} = \frac{1}{4(k+1)},  
\end{equation*}
and by condition~($Q_6$) for $n \geq r_{\beta}(11N(k+1)-1)$
\begin{equation*}
\frac{8}{3} \left| 1-\frac{\beta_{n+1}}{\beta_n} \right| \leq \frac{8N}{3(11N(k+1))} \leq \frac{1}{4(k+1)}.  
\end{equation*}
By Lemma~\ref{LemmaSuzuki2} and \eqref{e:34}, it follows that part~\eqref{qrcodd} holds.

For $n \geq N_1$, we have  $2N\left( \lambda_{n}+\lambda_{n-1}+ (\lambda_{n-1} +1)  \left | 1 - \frac{\mu_{n}}{\mu_{n-1}} \right |\right) \leq \frac{1}{2(k+1)}$. By part~\eqref{qrcodd}, there is $n_0 \leq \eta_0(2k+1,\widetilde{f}[N_1])$ such that 
\begin{equation*}
\forall n \in [n_0, \widetilde{f}[N_1](n_0)] \left( \norm{y_{2n+1}-y_{2n-1}} \leq \frac{1}{2(k+1)} \right).
\end{equation*}
With $n_1:= \max\{n_0,N_1\}\leq \eta_1(k,f)$, part~\eqref{qrceven} follows from \eqref{e:dif_pares} and the fact that $[n_1, f(n_1)] \subseteq [n_0, \widetilde{f}[N_1](n_0)]$.
\end{proof}
\begin{lemma}\label{l:asymreg2}
Let $N \in \N$ be such that $N \geq \max\{\norm{u-q}, \norm{x_0-q}, \norm{q}\}$, for some $q \in S$. Assume that there exist monotone functions $a,\ell, r_{\beta}, r_{\mu}: \N \to \N$ satisfying conditions $(Q_1), (Q_2), (Q_6)$ and $(Q_7)$. Then
 $ \norm{y_{n+1}-y_{n}} \to 0$ with monotone quasi-rate of convergence $$\eta_2(k,f):=\eta_2[N](k,f):=2\max\{\eta_1(8N(k+1)^2-1,\check{f}),N_2\}+1,$$
where $\eta_1$ is as in Lemma~\ref{l:asymreg1} and
\begin{enumerate}
\item[]  $\check{f}(n):=f(2\max\{n, N_2\}+1)$
\item[] $N_2:= \max\{a(32N^2(k+1)^2-1),\ell(32N^2(k+1)^2-1)\}$.
\end{enumerate}
\end{lemma}
\begin{proof}
By the definition of $(y_n)$ we have
\begin{equation*}
\beta_{n+1}\A y_{2n+1} \ni \alpha_n \left( u-y_{2n}\right) + y_{2n}-y_{2n+1}
\end{equation*}
Then by the monotonicity of $\A$
\begin{equation*}
\ip{ y_{2n+1}-y_{2n}-\alpha_n \left( u-y_{2n}\right),y_{2n+1}-q} \leq 0
\end{equation*}
Now, the latter implies that
\begin{equation*}
\begin{split}
\ip{ y_{2n+1}-y_{2n+2},y_{2n+1}-q} &\leq  \ip{ \alpha_n \left( u-y_{2n}\right)-y_{2n+2}+y_{2n},y_{2n+1}-q}\\
& \leq \norm{y_{2n+1}-q}\left(\alpha_n \norm{u-y_{2n}}+\norm{y_{2n+2}-y_{2n}}\right)\\
& \leq 2N^2\alpha_n+ N\norm{y_{2n+2}-y_{2n}}
\end{split}
\end{equation*}
Similarly,
\begin{equation*}
\mu_{n+1}\B y_{2n+2} \ni \lambda_n \left( u-y_{2n+1}\right) + y_{2n+1}-y_{2n+2}
\end{equation*}
Then, by the monotonicity of $\B$
\begin{equation*}
\ip{ y_{2n+2}-y_{2n+1}-\lambda_n \left( u-y_{2n+1}\right),y_{2n+2}-q} \leq 0
\end{equation*}
The latter implies that
\begin{equation*}
\begin{split}
\ip{ y_{2n+2}-y_{2n+1},y_{2n+2}-q} &\leq  \ip{ \lambda_n \left( u-y_{2n+1}\right),y_{2n+2}-q}\\
& \leq \lambda_n\norm{u-y_{2n+1}}\norm{y_{2n+2}-q}\\
& \leq 2N^2\lambda_n
\end{split}
\end{equation*}
Let $k,f$ be given. By Lemma~\ref{l:asymreg1} there exists $n_0\leq \eta_1(8N(k+1)^2-1,\check{f})$ such that $\norm{y_{2n+2}-y_{2n}}\leq \frac{1}{8N(k+1)^2}$, for all $n \in [n_0,\check{f}(n_0)]$. Define $n_1:= \max\{n_0, N_2\}$.
We have, for $n \in [n_1,f(2n_1+1)] \subseteq [n_0, \check{f}(n_0)]$
\begin{equation*}
\begin{split}
\norm{y_{2n+2}-y_{2n+1}}^2 &= \ip{ y_{2n+2}-y_{2n+1},y_{2n+2}-q} +\ip{ y_{2n+1} -y_{2n+2}, y_{2n+1}-q}\\
& \leq 2N^2 (\alpha_n +\lambda_n) +N\norm{y_{2n+2}-y_{2n}}\\
& \leq 2N^2\left(\frac{1}{32N^2(k+1)^2}+\frac{1}{32N^2(k+1)^2}\right)+N\frac{1}{8N(k+1)^2} \leq \frac{1}{4(k+1)^2},
\end{split}
\end{equation*}
and then $\norm{y_{2n+2}-y_{2n+1}} \leq \frac{1}{2(k+1)}$. With $n_2:= 2n_1+1\leq \eta_2(k,f)$, let $n \in [n_2,f(n_2)]$. 
If $n=2m+1$, then $m \in [n_1,f(2n_1+1)]$ and 
\begin{equation*}
\norm{y_{n+1}-y_n}=\norm{y_{2m+2}-y_{2m+1}}\leq \frac{1}{2(k+1)}\leq \frac{1}{k+1}.
\end{equation*}
If $n=2m$, then  again $m \in [n_1,f(2n_1+1)]\subseteq [n_0, \check{f}(n_0)]$ and
\begin{equation*}
\norm{y_{n+1}-y_n}\leq \norm{y_{2m+1}-y_{2m+2}}+\norm{y_{2m+2}-y_{2m}}\leq \frac{1}{2(k+1)}+ \frac{1}{8N(k+1)^2}\leq\frac{1}{k+1},
\end{equation*}
and the result holds.
\end{proof}
\begin{remark}\label{r:eta'}
Lemma~\ref{l:asymreg2} entails that the condition $(Q_{\eta'})$ is satisfied with $\eta' = \eta_2$. 
\end{remark}
\begin{lemma}\label{l:asymreg3}
Let $N \in \N$ be such that $N \geq \max\{\norm{u-q}, \norm{x_0-q}, \norm{q}\}$, for some $q \in S$. Assume that there exist $R \in \N\setminus \{0\}$ and monotone functions $a,\ell, r_{\beta}, r_{\mu}: \N \to \N$ satisfying conditions $(Q_1), (Q_2), (Q_4), (Q_6)$ and  $(Q_7)$. Then
\begin{enumerate}[$(i)$]
\item $ \|J^\A(y_{2n+1})-y_{2n+1}\| \to 0$ with monotone quasi-rate of convergence $\eta_3(k,f):=\max\{\eta_2(4k+3,\widetilde{f}[N_3]),N_3\}$
\item $ \|J^\B(y_{2n+2})-y_{2n+2}\|\to 0$ with monotone quasi-rate of convergence $\eta_4(k,f):=\max\{\eta_2(4k+3,\widetilde{f}[N_4]), N_4\}$
\end{enumerate}
where $\eta_2$ is as in Lemma~\ref{l:asymreg2} and
\begin{enumerate}
\item[] $N_3:=a(8N(k+1)-1)$
\item[] $N_4:=\ell(8N(k+1)-1)$.
\end{enumerate}
\end{lemma}
\begin{proof}
For all $n \in \N$ we have,
\begin{equation}\label{e:JA}
\begin{split}
\|J^\A_{\beta_n}(y_{2n+1})-y_{2n+1}\|&\leq \|J^\A_{\beta_n}(y_{2n+1})-J^\A_{\beta_n}(y_{2n})\|+\|J^\A_{\beta_n}(y_{2n})-y_{2n+1}\|\\
&\leq \|y_{2n+1}-y_{2n}\|+\|J^\A_{\beta_n}(y_{2n})-J^\A_{\beta_n}( \alpha_nu+(1-\alpha_n)y_{2n} )\|\\
&\leq \|y_{2n+1}-y_{2n}\|+\alpha_n\|u-y_{2n}\|\\
& \leq \|y_{2n+1}-y_{2n}\|+2N\alpha_n.
\end{split}
\end{equation}
By Lemma~\ref{l:asymreg2} there exists $n_0 \leq \eta_2(4k+3,\widetilde{f}[N_3])$ such that $\|y_{2n+1}-y_{2n}\|\leq \frac{1}{4(k+1)}$, for all $n \in [n_0, \widetilde{f}[N_3](n_0)]$. With $n_1 := \max\{n_0, N_3\} \leq \eta_3(k,f)$, condition $(Q_4)$, Lemma~\ref{lemmaresolvineq} and \eqref{e:JA} entail that, for all $n \in [n_1,f(n_1)]$
\begin{equation*}
\|J^\A(y_{2n+1})-y_{2n+1}\|\leq 2\|J^\A_{\beta_n}(y_{2n+1})-y_{2n+1}\|\leq \frac{1}{k+1}.
\end{equation*}  
The proof of the second part of the lemma is similar using
\begin{align*}
\|J^\B_{\mu_n}(y_{2n+2})-y_{2n+2}\|&\leq \|J^\B_{\mu_n}(y_{2n+2})-J^\B_{\mu_n}(y_{2n+1})\|+\|J^\B_{\mu_n}(y_{2n+1})-y_{2n+2}\|\\
&\leq \|y_{2n+2}-y_{2n+1}\|+\|J^\B_{\mu_n}(y_{2n+1})-J^\B_{\mu_n}( \lambda_nu+(1-\lambda_n)y_{2n+1} )\|\\
&\leq \|y_{2n+2}-y_{2n+1}\|+\lambda_n\|u-y_{2n+1}\| \\
&\leq \|y_{2n+2}-y_{2n+1}\|+2N\lambda_n. \qedhere
\end{align*}
\end{proof}
From Lemmas~\ref{l:asymreg2} and \ref{l:asymreg3}, using Remark~\ref{r:metameta}, we obtain that the sequence $(y_n)$ is asymptotically regular with respect to both $J^{\A}$ and $J^{\B}$.
\begin{lemma}\label{l:asymfinal}
Let $N \in \N$ be such that $N \geq \max\{\norm{u-q}, \norm{x_0-q}, \norm{q}\}$, for some $q \in S$.  Assume that there exist $R \in \N\setminus \{0\}$ and monotone functions $a,\ell, r_{\beta}, r_{\mu}: \N \to \N$ satisfying conditions $(Q_1), (Q_2), (Q_4),(Q_6)$ and $(Q_7)$. Then
\begin{enumerate}[$(i)$]
\item $\norm{J^{A}(y_{n})-y_{n}} \to 0$, with monotone quasi-rate of convergence $\eta_5(k,f):= \Psi[\Phi[\widetilde\eta_2,\widetilde\eta_3]](k,f)$
\item $\norm{J^{B}(y_{n})-y_{n}} \to 0$, with monotone quasi-rate of convergence $\eta_6(k,f):=\Psi[\Phi[\widetilde\eta_2,\widetilde\eta_4]](k,f)$,
\end{enumerate}
where $\Phi$ is as in Remark~\ref{r:metameta}, $\Psi$ is as in Proposition~\ref{p:metameta2} and
\begin{enumerate}
\item[] $\widetilde\eta_2(k,f):=\eta_2(4k+3,2f)$
\item[] $\widetilde\eta_3(k,f):=\eta_3(2k+1,f)$
\item[] $\widetilde\eta_4(k,f):=\eta_4(2k+1,h_f)+1$
\item[] $h_f(n):=f(n+1)$.
\end{enumerate}
\end{lemma}
\begin{proof}
Since $J^\A$ is nonexpansive 
\begin{equation*}
\begin{split}
\norm{J^\A (y_{2n})-y_{2n}} &\leq \norm{J^\A (y_{2n})-J^\A (y_{2n+1})}+ \norm{J^\A (y_{2n+1})-y_{2n+1}}+\norm{y_{2n+1}-y_{2n}}\\
& \leq 2 \norm{y_{2n+1}-y_{2n}}+  \norm{J^\A (y_{2n+1})-y_{2n+1}}.
\end{split}
\end{equation*}
By Lemma~\ref{l:asymreg2} and Lemma~\ref{l:asymreg3} we may apply Remark~\ref{r:metameta} to conclude that for all $k \in \N$ and monotone $f:\N \to \N$ 
\begin{equation*}
\exists  n\leq \Phi[\widetilde\eta_2,\widetilde\eta_3](k,f) \, \forall m \in [n,f(n)] \left(\norm{J^\A (y_{2m})-y_{2m}}\leq \frac{1}{k+1} \wedge \norm{J^\A (y_{2m+1})-y_{2m+1}}\leq \frac{1}{k+1}\right).
\end{equation*}
The first part follows from Proposition~\ref{p:metameta2}.

The proof of the second part is analogous using
\begin{equation*}
\begin{split}
\norm{J^\B (y_{2n+1})-y_{2n+1}} &\leq \norm{J^\B (y_{2n+1})-J^\B (y_{2n})}+ \norm{J^\B (y_{2n})-y_{2n}}+\norm{y_{2n}-y_{2n+1}}\\
& \leq 2 \norm{y_{2n+1}-y_{2n}}+  \norm{J^\B (y_{2n})-y_{2n}}.\qedhere
\end{split}
\end{equation*}
\end{proof}
\begin{remark}\label{r:eta}
Using Remark~\ref{r:metameta}, we obtain from Lemma~\ref{l:asymfinal} that the function  $\eta(k,f):= \Phi[\eta_5,\eta_6](k,f)$ satisfies the condition $(Q_{\eta})$, i.e.\
\begin{equation*}\label{e:etasat}
\forall k \in \N\, \forall f \in \N \to \N \, \exists n \leq \eta(k,f) \, \forall m \in [n,f(n)] \left(\norm{J^{A}(y_{m})-y_{m}}\leq \frac{1}{k+1} \wedge  \norm{J^{B}(y_{m})-y_{m}} \leq \frac{1}{k+1} \right).
\end{equation*}
Then, using also Remark~\ref{r:eta'}, we may apply Corollary~\ref{c:quantBM} to obtain an effective metastability bound for \eqref{HPPA2} under the assumptions $(Q_1)-(Q_7)$.
\end{remark}
%
\subsection{Relation with Halpern's definition}\label{s:connecting}

In this subsection we argue that the algorithm \eqref{HPPA2} is essentially a Halpern-type iteration for two maximal monotone operators $\A',\B'$. 
Define for $n \in \N$
\begin{equation}\label{HPPAalt}\tag{\textsf{HPPA}$_2^{\star}$}
\begin{cases}
z_{2n+1}= \widetilde{\alpha}_n u+(1-\widetilde{\alpha}_n) J_{\widetilde{\beta}_n}^{\A'} \left(z_{2n}\right)\\ 
z_{2n+2}= \widetilde{\lambda}_n u+(1-\widetilde{\lambda}_n) J_{\widetilde{\mu}_n}^{\B'} \left(z_{2n+1}\right)
\end{cases}
\end{equation}
where $z_0 \in H$ is given, $(\widetilde{\alpha}_n), (\widetilde{\lambda}_n) \subset (0,1) $, $(\widetilde{\beta}_n),(\widetilde{\mu}_n) \subset (0,+\infty)$. 
\begin{proposition}\label{MarHalpern}
Let $(y_n)$ be generated by \eqref{HPPA2} with $y_n \to x \in H$. Consider $\A'=\A, \B'=\B$ and that $\widetilde{\alpha}_n=\lambda_n , \widetilde{\lambda}_n= \alpha_{n+1}, \widetilde{\beta}_n=\beta_n, \widetilde{\mu}_n = \mu_n$, for all $n \in \N$. If $\alpha_n, \lambda_n \to 0$, and $(z_n)$ is generated by \eqref{HPPAalt} with $z_0 = \alpha_0 u+(1-\alpha_0)y_0$, then $ z_n  \to x$. 

Moreover, given rates of convergence $a, \ell:\N \to \N$ for $\alpha_n \to 0$ and $\lambda_n \to 0$ (respectively), and $N \in \N$ such that $N \geq  \max\{\norm{u-q}, \norm{x_0-q}, \norm{q}\}$, for some $q \in S$, we have that $\norm{y_n-z_n} \to$ with rate of convergence $\gamma(k):=2\max\{\widetilde{a}(k),\widetilde{\ell}(k)\}+2$. If $\rho$ is a rate of metastability for $(y_n)$, then $\widetilde{\rho}$ is a rate of metastability for $(z_n)$, where  
$$\widetilde{\rho}(k,f):=\max\left \{\rho\left(2k+1,\widetilde{f}[\gamma(4k+3)]\right), \gamma(4k+3)\right\},$$ 
\begin{enumerate}
\item[]  $\widetilde{a}(k):=a(2N(k+1)-1)$
\item[] $\widetilde{\ell}(k):=\ell(2N(k+1)-1)$
\end{enumerate}
\end{proposition}
\begin{proof}
By induction, $y_{2n+1}=J^{\A'}_{\widetilde{\beta}_n}(z_{2n})$ and $y_{2n+2}=J^{\B'}_{\widetilde{\mu}_n}(z_{2n+1})$, for all $n \in \N$. Hence 
\begin{equation}\label{e:znyn}
\begin{cases}
z_{2n+1}=\widetilde{\alpha}_nu+(1-\widetilde{\alpha}_n)y_{2n+1}\\
z_{2n+2}=\widetilde{\lambda}_nu+(1-\widetilde{\lambda}_n)y_{2n+2}
\end{cases}
\end{equation}
Since $\widetilde{\alpha}_n=\lambda_n \to 0$ and $\widetilde{\lambda}_n= \alpha_{n+1} \to 0$, we conclude $z_n \to x$.
Now let $k, f$ be given. From \eqref{e:znyn}, we have that $\norm{z_{2n+1}-y_{2n+1}} \to 0$ with rate of convergence $\widetilde{\ell}$ and $\norm{z_{2n+2}-y_{2n+2}} \to 0$ with rate of convergence $\widetilde{a}$. This shows that $\gamma$ is a rate of convergence for $\norm{z_n -y_n} \to 0$. There exists $n_0 \leq \rho(2k+1,\widetilde{f}[\gamma(4k+3)])$ such that 
\begin{equation*}
\forall i,j \in [n_0,\widetilde{f}[\gamma(4k+3)](n_0)] \left(\norm{y_i-y_j} \leq \frac{1}{2(k+1)}\right).
\end{equation*}
The result follows with $n:=\max\{n_0, \gamma(4k+3)\} \leq \widetilde{\rho}(k,f)$ by the triangle inequality.
\end{proof}
\begin{proposition}\label{MarHalpern2}
Let $(z_n)$ be generated by \eqref{HPPAalt} with $z_n \to x \in H$. Consider $\A=\B', \B=\A'$ and that $\alpha_n=\widetilde{\alpha}_n, \lambda_n=\widetilde{\lambda}_n, \beta_n=\widetilde{\mu}_n, \mu_n=\widetilde{\beta}_{n+1}$, for all $n \in \N$. If $\widetilde{\alpha}_n, \widetilde{\lambda}_n \to 0$, and $(y_n)$ is generated by \eqref{HPPA2} with $y_0 = J^{\A'}_{\widetilde{\beta}_0}(z_0)$, then $ y_n \to x$. 

Moreover, given rates of convergence $a, \ell:\N \to \N$ for $\widetilde{\alpha}_n \to 0$ and $\widetilde{\lambda}_n \to 0$ (respectively), and $N \in \N$ such that $N \geq  \max\{\norm{u-q}, \norm{x_0-q}, \norm{q}\}$, for some $q \in S$, we have that $\norm{z_{n+1}-y_n} \to 0$ with rate of convergence $\gamma(k):=2\max\{\widetilde{a}(k),\widetilde{\ell}(k)\}+1$. If $\rho$ is a rate of metastability for $(z_n)$, then $\widetilde{\rho}$ is a rate of metastability for $(y_n)$, where 
$$\widetilde{\rho}(k,f):=\max\left \{\rho\left(2k+1,\widetilde{f}[\gamma(4k+3)]+1\right), \gamma(4k+3)\right\},$$ 
with $\widetilde{a}, \widetilde{\ell}$ defined as in Proposition~\ref{MarHalpern}.
%
\end{proposition}
\begin{proof}
By induction, for all $n \in \N$ 
\begin{equation}\label{e:znyn2}
\begin{cases}
z_{2n+1}=\alpha_nu+(1-\alpha_n)y_{2n}\\
z_{2n+2}=\lambda_nu+(1-\lambda_n)y_{2n+1}
\end{cases}
\end{equation}
which implies $y_n \to x$.
Now let $k, f$ be given. From \eqref{e:znyn2}, we have that $\norm{z_{2n+1}-y_{2n}} \to 0$ with rate of convergence $\widetilde{a}$, and $\norm{z_{2n+2}-y_{2n+1}} \to 0$ with  rate of convergence $\widetilde{\ell}$. This shows that $\gamma$ is a rate of convergence for $\norm{z_{n+1} -y_n} \to 0$. There exists $n_0 \leq \rho(2k+1,\widetilde{f}[\gamma(4k+3)]+1)$ such that 
\begin{equation*}
\forall i,j \in [n_0,\widetilde{f}[\gamma(4k+3)](n_0)+1] \left(\norm{z_i-z_j} \leq \frac{1}{2(k+1)}\right),
\end{equation*}
which entails $\norm{z_{i+1}-z_{j+1}} \leq \dfrac{1}{2(k+1)}$, for $i,j \in  [n_0,\widetilde{f}[\gamma(4k+3)](n_0)]$. Again the result follows by the triangle inequality, with $n:=\max\{n_0, \rho'(4k+3)\} \leq \widetilde{\rho}(k,f)$.
\end{proof}
\begin{remark}\label{r:rar}
\begin{enumerate}[$(i)$]
\item In the first part of Proposition~\ref{MarHalpern2}, if $p \in S':=\A'^{-1}(0) \cap \B'^{-1}(0)$, then we can drop the assumptions $\widetilde{\alpha}_n,\widetilde{\lambda}_n \to 0$. Indeed, from \eqref{e:znyn2}, one has $y_{2n+1}=J^{\A}_{\beta_n}(z_{2n+1})$ and  $y_{2n+2}=J^{\B}_{\mu_n}(z_{2n+2})$, for all $n \in \N$. Since $p \in S'$ and the resolvent functions are nonexpansive,
\begin{equation*}
\norm{y_{2n+1}-x}\leq \norm{z_{2n+1}-x} \quad \mbox{and} \quad \norm{y_{2n+2}-x}\leq \norm{z_{2n+2}-x},
\end{equation*}
from which we conclude $y_n \to x$. Moreover, a (quasi-)rate of convergence for $z_n \to x$ easily gives rise to a (quasi-)rate of convergence for $y_n \to x$. 
\item Let $r$ be a positive real number. 
In Proposition~\ref{MarHalpern}, if $\rho$ is a quasi-rate of asymptotic regularity for $(y_n)$ w.r.t.\ $J^\A_{r}$ and $J^\B_{r}$, i.e.\ a quasi-rate of convergence for $\max\{ \norm{y_n-J^\A_{r}(y_n)}, \norm{y_n-J^\B_{r}(y_n)}\} \to 0$, then $\widetilde{\rho}$ is a quasi-rate of asymptotic regularity for $(z_n)$, using the inequality $$\norm{z_n-J_{r}(z_n)} \leq 2 \norm{z_{n}-y_{n}}+ \norm{y_n-J_{r}(y_n)},$$ where $J_{r}$ denotes either $J^\A_{r}$ or $J^\B_{r}$.
\item Similarly to the previous point, in Proposition~\ref{MarHalpern2} the same transformations yield  a quasi-rate of asymptotic regularity for $(y_n)$ from a quasi-rate of asymptotic regularity for $(z_n)$, now using the inequality $\norm{y_n-J_{r}(y_n)} \leq 2 \norm{z_{n+1}-y_{n}}+ \norm{z_{n+1}-J_{r}(z_{n+1})}$. 
\item The previous arguments actually entail that if $\rho$ is a Cauchy rate (resp. rate of asymptotic regularity) for one iteration, then $\widetilde{\rho}$ is also a Cauchy rate (resp. rate of asymptotic regularity)  for the other iteration. 
\item In light of Proposition~\ref{MarHalpern} (and point (ii) in this remark), the extraction of metastability rates and quasi-rates of asymptotic regularity for \eqref{HPPA2}, carried out in this paper, entail similar quantitative information for \eqref{HPPAalt}. 
\end{enumerate} 
\end{remark}
\section{Improving the metastability for \texorpdfstring{\eqref{HPPA2}}{MAR*}}\label{s:second}

In a follow up paper \cite{BM(12)} Boikanyo and Moro\c{s}anu improved Theorem~\ref{t:BM_exact} by removing conditions ($C_4$) and ($C_5$).  In this section we give an analysis of the improved result. In order to obtain a quantitative version of this improved result, we work with conditions $(Q_1)-(Q_4)$ and, additionally, assume the existence of a monotone function $\mathcal{P}:\N \to  \N\setminus \{0\}$ satisfying 
\begin{equation}\tag{$Q_8$}
\forall n \in \N \left(\alpha_n+\lambda_n \geq \frac{1}{\mathcal{P}(n)}\right).
\end{equation}
From a logical point of view, the proof is now more involved since it requires a discussion by cases depending on whether the sequence defined by $s_n^x:=\norm{y_{2n}-x}^2$, for $x \in B_N$ is eventually decreasing or not and uses a (quantitative version of a) lemma due to Maingé \cite[Lemma~3.1]{M(08)}. The latter is often used in results which relax the conditions on the parameters of the resolvent functions.
The first quantitative analyses using Maingé's result were developed in \cite{DP(21),Koh(ta)}.

Throughout this section we fix $a,\ell,A, R$ and $\mathcal{P}$ satisfying conditions $(Q_1) - (Q_4)$ and $(Q_8)$ respectively, and $N \in \N$.
Let us start by defining the following constants and functions, which are useful for our analysis.
\begin{definition}\label{d:	functions}
Given $k \in \N$, we define the following constants
\begin{enumerate}[$(i)$]
\item (cf.\ Corollary~\ref{c:metatoasymp}) $M_0:=\max\{a(16N(k+1)-1),\ell(16N(k+1)-1)\}$
\item (cf.\ Lemma~\ref{l:conseqsmall}) $M_1:=\max\{a((8N(k+1))^2)-1, \ell(3(4N(k+1))^2)-1\}$
\item (cf.\ Lemma~\ref{l:asymregcase1}) $M_2:=\max\{a((64N(k+1))^2)-1, \ell(3(32N(k+1))^2)-1\}$
\end{enumerate}
\end{definition}
\begin{definition}\label{d:	functions}
	We define the following functions 
	\begin{enumerate}[$(i)$]
	\item  (cf. Lemma~\ref{l:metadifference}) Given $D,k, n \in \N$ and $f:\N \to \N$
\begin{enumerate}
	 \item $\phi_1[D](k,n,f):=f(\phi_2[D](k,n,f))+1$
	 \item $\phi_2[D](k,n,f):= \max\{n, (f+1)^{(D(k+1))}(n)\}$
\end{enumerate}
	\item (cf. Lemma~\ref{l:conseqsmall}) Given $k, n \in \N$ and $f:\N \to \N$
\begin{enumerate}
	 \item $\Psi_1(k,n,f):=\phi_1[4N^2](2(k+1)^2-1,n,\widetilde{{f}}[M_1])$
	 \item $\Psi_2(k,n,f):=\max\{\phi_2[4N^2](2(k+1)^2-1,n,\widetilde{f}[M_1]),M_1\}$
\end{enumerate}
	\item (cf. Lemma~\ref{l:asymregcase1}) Given $k, n \in \N$ and $f:\N \to \N$
\begin{enumerate}
	 \item $\Psi_3(k,n,f):= \phi_1[4N^2](128(k+1)^2-1,n,\widetilde{{f}}[M_2])$
	 \item $\Psi_4(k,n,f):=\max\{\phi_2[4N^2](128(k+1)^2-1,n,\widetilde{f}[M_2]),M_2\}$
\end{enumerate}
	 \item  (cf. Lemma~\ref{l:case1}) Let $\sigma_1:=\sigma_1[A,4N^2]$ be as in Lemma~\ref{l:Xuquant1}. Given $k, n \in \N$, $f:\N \to \N$  
\begin{enumerate}
	\item $j_{f}[k,f](n):=f(2\sigma_1(\widetilde{k},n)+1)$, with $\widetilde{k}\equiv16(k+1)^2-1$
	\item $\Psi_5(k,n,f):= \Psi_3(\max\{k,n\},n,j_f[k,f])$
	\item $\Psi_6(k,n,f):=2\sigma_1(\widetilde{k},\widetilde{n})+1$, with $\widetilde{n}\equiv \Psi_4(\max\{k,n\},n,j_f[k,f])$ and $\widetilde{k}$ as in (a)
	\item $\xi_1[k,f](n):= 16\cdot64N(k+1)^2(f(\Psi_6(k,n,f))+1)-1$
\end{enumerate}
	 \item  (cf. Lemma~\ref{lemmamaincase2}) Given $k, n \in \N$ and $f:\N \to \N$
\begin{enumerate}
	\item $\xi_2[k,f](n):=16\cdot96N\mathcal{P}(f(2n+1))(k+1)^2-1$
	\item $\mathfrak{m}(n):=\max\{a(16\cdot64N^2n-1), \ell(12\cdot64N^2n-1)\}$
 	\item $r(n,k):=\max\{\mathfrak{m}(2(n+1)^2), \mathfrak{m}((k+1)^2)\}$
\end{enumerate}
	 \item (cf. Theorem~\ref{t:quantBM2}) Let $\zeta=\zeta[N]$ be as in Lemma~\ref{l:projection}. For all $k, n\in \N$ and $f:\N \to \N$
\begin{enumerate}
	\item $r^{+}(n,k):=\max\{ r(n,k), n\}$
	\item $\Xi_1[k,f](n):=\xi_1[k,f](r^{+}(n,k))$
	\item $\Xi_2[k,f](n):=\xi_2[k,f](\Psi_5(k, r^{+}(n,k),f))$
	\item $\Xi(n):=\Xi[k,f](n):=\Xi[k,f](n):=\max\{\Xi_1[k,f](n), \Xi_2[k,f](n)\}$
	\item $\mu(k,f):= \max \{\Psi_6(k,r^{+}(\zeta(\overline{k},\Xi),k),f), 2\Psi_5(k, r^{+}(\zeta(\overline{k},\Xi), k),f)+1\}$, with $\overline{k}=3\cdot 96(k+1)^2-1$.
\end{enumerate}
	\end{enumerate}
\end{definition}
We are now able to formulate the main result of this section.
\begin{theorem}\label{t:quantBM2}
For $x_0, u \in H $, let $(y_n)$ be generated by \eqref{HPPA2e} (with $y_0=x_0$). Consider $N \in\N \setminus \{0\}$ such that $N \geq \max\{ 2\norm{u-q}, \norm{x_0-q}, \norm{q}\}$, for some $q \in S$. Then, for all $k \in \N$ and monotone function $f:\N\to \N$
\begin{equation*}
\exists n \leq \mu(k,f)\, \forall i,j\in [n,f(n)] \left(\norm{y_{i}-y_j}\leq \frac{1}{k+1} \right).
\end{equation*}
%
%
\end{theorem}
From the metastability property of Theorem~\ref{t:quantBM2} it is possible to show the following asymptotic regularity result. We write $J^\A:=J^\A_{R^{-1}}$ and $J^\B:=J^\B_{R^{-1}}$.

\begin{corollary}\label{c:metatoasymp}
The sequence $(y_n)$ is asymptotically regular w.r.t.\ $J^{\A}$ and $J^{\B}$,
with quasi-rate of asymptotic regularity $\vartheta(k,f):=2\max\{\mu(8k+7,2\widetilde{g_{f}}[M_0]+2),M_0\}+1$.
\end{corollary}
\begin{proof}
Let  $k \in \N$ and monotone $f:\N \to \N$ be given. By Theorem~\ref{t:quantBM2}, there exists $n_0 \leq \mu(8k+7,2\widetilde{g_{f}}[M_0]+2)$ such that 
\begin{equation*}
\forall i,j \in [n_0,2\widetilde{g_f}[M_0](n_0)+2] \left(\norm{y_{i}-y_j} \leq \frac{1}{8(k+1)}\right).
\end{equation*}
Define $n_1:=\max\{n_0,M_0\}$. For all $i \in [n_1,g_f(n_1)]$, we have $2i,2i+1,2i+2 \in [n_0,2\widetilde{g_f}[M_0](n_0)+2]$ and then 
\begin{equation*}
\norm{y_{2i+2}-y_{2i+1}} \leq \frac{1}{8(k+1)} \quad \mbox{and} \quad \norm{y_{2i+1}-y_{2i}} \leq \frac{1}{8(k+1)}.
\end{equation*}
From the definition of \eqref{HPPA2}, using $\norm{y_{2i}-u}\leq 2N$ and the fact that  $J^\A_{\beta_i}$ are nonexpansive, we derive
\begin{equation*}
\begin{split}
\norm{J^\A_{\beta_i} (y_{2i})-y_{2i}} &\leq \norm{J^\A_{\beta_i} (y_{2i})-y_{2i+1}}+ \frac{1}{8(k+1)}\\
& \leq \alpha_i \norm{y_{2i}-u}+ \frac{1}{8(k+1)} \\
& \leq \frac{2N}{16N(k+1)}+ \frac{1}{8(k+1)} = \frac{1}{4(k+1)}.
\end{split}
\end{equation*}
Similarly, one obtains
\begin{equation*}
\norm{J^\B_{\mu_i} (y_{2i+1})-y_{2i+1}} \leq \frac{1}{4(k+1)}.
\end{equation*}
By Lemma~\ref{lemmaresolvineq} and condition $(Q_4)$ we conclude
\begin{equation*}
\norm{J^\A(y_{2i})-y_{2i}}\leq  \frac{1}{2(k+1)} \quad \mbox{and}\quad \norm{J^\B(y_{2i+1})-y_{2i+1}} \leq \frac{1}{2(k+1)}.
\end{equation*}
Finally, since $J^\A$ and $J^\B$ are nonexpansive, we have for $i \in [n_1,g_f(n_1)]$ 
\begin{equation*}
\norm{J^\A (y_{2i+1})-y_{2i+1}} \leq 2 \norm{y_{2i+1}-y_{2i}}+  \norm{J^\A (y_{2i})-y_{2i}} \leq \frac{1}{k+1}
\end{equation*}
\begin{equation*}
\norm{J^\B (y_{2i})-y_{2i}} \leq 2 \norm{y_{2i+1}-y_{2i}}+  \norm{J^\B (y_{2i+1})-y_{2i+1}} \leq \frac{1}{k+1}.
\end{equation*}
With $n_2:=2n_1+1\leq \vartheta(k,f)$ the result follows by considering the cases $m=2i$ and $m=2i+1$, for $i \in [n_2,f(n_2)]$ as in both cases $i \in [n_1,g_f(n_1)]$. 
\end{proof}

In order to prove Theorem~\ref{t:quantBM2} we start with the following inequalities. Let $x\in B_N$. For all $m \in \N$, by the definition of $y_{2m+2}$ we have
\begin{equation*}
\B(y_{2m+2}) \ni \frac{1}{\mu_m}\left(\lambda_m(u-x)+(1-\lambda_m)(y_{2m+1}-x)-(y_{2m+2}-x) \right).
\end{equation*}

We have that
\begin{equation*}
\max \left\{R\norm{J^{\B}(x)-y_{2m+2}},\norm{\frac{1}{\mu_m}\left(\lambda_mu+(1-\lambda_m)y_{2m+1}-y_{2m+2} \right)}\right\} \leq 2RN.
\end{equation*}
By Lemma~\ref{l:monotone1} 
\begin{equation*}
\ip{ y_{2m+2}-x,\frac{1}{\mu_m}\left(\lambda_m(u-x)+(1-\lambda_m)(y_{2m+1}-x)-(y_{2m+2}-x) \right)} \geq -4RN\norm{J^{\B}(x)-x},
\end{equation*}
and then, using the fact that $\mu_m R\leq 1$ and the equality $2\ip{ a,b}=\norm{a}^2+\norm{b}^2-\norm{a-b}^2$,
\begin{equation*}
\begin{split}
2\norm{y_{2m+2}-x}^2 &\leq 2\lambda_m\ip{ y_{2m+2}-x, u-x} +2(1-\lambda_m)\ip{ y_{2m+2}-x,y_{2m+1}-x} + 8N\norm{J^{\B}(x)-x}\\
& = (1-\lambda_m)\left(\norm{y_{2m+2}-x}^2+\norm{y_{2m+1}-x}^2 -\norm{y_{2m+2}-y_{2m+1}}^2\right)\\
&\quad+2\lambda_m\ip{ y_{2m+2}-x, u-x} + 8N\norm{J^{\B}(x)-x}.
\end{split}
\end{equation*}
Hence
\begin{equation}\label{e:main1}
\begin{split}
(1+\lambda_m)\norm{y_{2m+2}-x}^2&\leq (1-\lambda_m)\norm{y_{2m+1}-x}^2+2\lambda_m \ip{ u-x, y_{2m+2}-x } \\
&\quad - (1-\lambda_m)\norm{y_{2m+2}-y_{2m+1}}^2+ 8N\norm{J^{\B}(x)-x}.
\end{split}
\end{equation}
Using 
\begin{equation*}
\A(y_{2m+1}) \ni \frac{1}{\beta_m}\left(\alpha_m(u-x)+(1-\alpha_m)(y_{2m}-x)-(y_{2m+1}-x) \right),
\end{equation*}
with similar arguments we obtain
\begin{equation}\label{e:main2}
\begin{split}
(1+\alpha_m)\norm{y_{2m+1}-x}^2&\leq (1-\alpha_m)\norm{y_{2m}-x}^2+2\alpha_m \ip{ u-x, y_{2m+1}-x } \\
&\quad - (1-\alpha_m)\norm{y_{2m+1}-y_{2m}}^2+ 8N\norm{J^{\A}(x)-x}.
\end{split}
\end{equation}
From \eqref{e:main1} and \eqref{e:main2} we obtain
\begin{equation}\label{e:main3}
\begin{split}
(1+\lambda_m)\norm{y_{2m+2}-x}^2&\leq (1-\alpha_m)(1-\lambda_m)\norm{y_{2m}-x}^2+2\alpha_m (1-\lambda_m)\ip{ u-x, y_{2m+1}-x } \\
& \quad +2\lambda_m \ip{ u-x, y_{2m+2}-x }- (1-\alpha_m)(1-\lambda_m)\norm{y_{2m+1}-y_{2m}}^2 \\
&\quad - (1-\lambda_m)\norm{y_{2m+2}-y_{2m+1}}^2+ 8N\left(\norm{J^{\A}(x)-x}+\norm{J^{\B}(x)-x}\right).
\end{split}
\end{equation}
Denote $s_n^x:=\norm{y_{2n}-x}^2$, for $x \in B_N$. It follows from Lemma~\ref{l:vbounded} and \eqref{e:main3}
\begin{equation}\label{e:main4}
\begin{split}
s_{m+1}^x-s_m^x &+ \norm{y_{2m+1}-y_{2m}}^2 + \norm{y_{2m+2}-y_{2m+1}}^2 
\\ &\leq \alpha_m \left(2\norm{u-x}\norm{y_{2m+2}-x}+\norm{y_{2m+2}-y_{2m+1}}^2+\norm{y_{2m+1}-y_{2m}}^2\right)\\
& \quad +\lambda_m \left(2\norm{u-x}\norm{y_{2m+1}-x}+\norm{y_{2m+1}-y_{2m}}^2\right)\\
&\leq 16N^2\alpha_m  + 12N^2\lambda_m.
\end{split}
\end{equation}
On the other hand, from the definition of $(y_{n})$, condition ($Q_4$) and Lemmas~\ref{lemmaresolvineq} and \ref{l:vbounded},
\begin{equation}\label{e:main5}
\begin{split}
\norm{J^{\B}(y_{2m+2})-y_{2m+2}} &\leq 2\norm{J^{\B}_{\mu_m}(y_{2m+2})-y_{2m+2}}\\
&\leq 2 \norm{\lambda_m (u-y_{2m+1}) + (y_{2m+1}-y_{2m+2})}\\
& \leq  4N\lambda_m + 2\norm{y_{2m+2}-y_{2m+1}}
\end{split}
\end{equation}
and
\begin{equation}\label{e:main6}
\norm{J^{\A}(y_{2m+1})-y_{2m+1}} \leq  4N\alpha_m + 2\norm{y_{2m+1}-y_{2m}}.
\end{equation}

The proof of Theorem~\ref{t:quantBM2} now follows a discussion by cases.
\subsection{First case}\label{s:First}

Let us consider first the case where $(s_n^x)$ is eventually decreasing. In this subsection we assume that  for some $q \in S$ it holds that $N \geq \max\{\norm{u-q}, \norm{x_0-q}, \norm{q}\}$.
We begin with an easy adaptation of \cite[Proposition 2.27]{K(08)} (see also Remark~2.29 in the same reference).
\begin{lemma}\label{l:metadifference}
		Let $(s_n)$ be a sequence of real numbers and $D \in \N \setminus \{ 0 \}$ be such that $0 \leq s_m \leq D$, for all $m \in \N$. For $k,n \in \N$, $f:\N \to \N$ monotone, if
	\begin{equation*}
	\forall m \in [n, \phi_1(k,n,f)] \left(s_{m+1} \leq s_{m} \right),
	\end{equation*}
	then there exists $n' \leq \phi_2 (k,n,f)$ such that
	\begin{equation}\label{eqCauchymeta2}
	\forall m \in [n',f(n')] \left(s_{m}-s_{m+1}\leq \frac{1}{k+1}\right).
	\end{equation}
	Moreover, there is $n' \in \{(f+1)^{(i)}(n): i \leq D(k+1)\}$ satisfying \eqref{eqCauchymeta2}.
\end{lemma}
From Lemma~\ref{l:metadifference} it follows the following result tailored for our analysis.
\begin{lemma}\label{l:conseqsmall}
For all $k,n \in \N$ and monotone $f:\N \to \N$, if there exists $x \in B_N$ such that  
\begin{equation*}
\forall m \in \left[n, \Psi_1(k,n,f)\right] \left(s_{m+1}^{x}\leq s_{m}^{x}\right),
\end{equation*}
then there exists $n' \leq \Psi_2(k,n,f)$ such that 
\begin{equation*}
\forall m \in [n',f(n')] \left(\norm{y_{2m+2}-y_{2m+1}} \leq \frac{1}{k+1} \wedge \norm{y_{2m+1}-y_{2m}} \leq \frac{1}{k+1}\right).
\end{equation*}
\end{lemma}
\begin{proof}
First note that, using Lemma~\ref{l:vbounded}, it holds that $0 \leq s_{m}^{x}\leq 4N^2$, for all $m \in \N$. By Lemma~\ref{l:metadifference} there exists $n_0 \leq \phi_2[4N^2](2(k+1)^2-1,n,\widetilde{f}[M_1])$ such that for all $m \in [n_0,\widetilde{f}[M_1](n_0)]$ it holds that $s_{m}^{x}-s_{m+1}^{x}\leq \dfrac{1}{2(k+1)^2}$. By \eqref{e:main4} the result follows with $n':= \max\{n_0,M_1\} \, \left(\leq \Psi_2(k,n,f)\right)$. 
\end{proof}
\begin{lemma}\label{l:asymregcase1}
 For all $k,n \in \N$ and monotone $f:\N \to \N$, if there exists $x \in B_N$ such that   
\begin{equation*}
\forall m \in [n, \Psi_3(k,n,f)] \left(s_{m+1}^{x}\leq s_{m}^{x}\right),
\end{equation*}
then there exists $n' \leq \Psi_4(k,n,f)$ such that for all $i \in [n',f(n')]$
\begin{equation*}
\begin{split}
& \norm{J^\A (y_{2i+2})-y_{2i+2}}\leq \frac{1}{k+1} \wedge \norm{J^\A (y_{2i+1})-y_{2i+1}}\leq \frac{1}{k+1}\\
\mbox{and} \qquad &\\
& \norm{J^\B (y_{2i+2})-y_{2i+2}}\leq \frac{1}{k+1} \wedge \norm{J^\B (y_{2i+1})-y_{2i+1}}\leq \frac{1}{k+1}.
\end{split}
\end{equation*}
%
\end{lemma}
\begin{proof}
By Lemma~\ref{l:conseqsmall}, there exists $n_0 \leq \Psi_2(8k+7,n,\widetilde{f}[M_2])$ such that for $n':= \max\{n_0,M_2\}$,  
\begin{equation}\label{e:normless8}
\forall i \in [n',f(n')]\left(\max\{\norm{y_{2i+1}-y_{2i}}, \norm{y_{2i+2}-y_{2i+1}}\}\leq \frac{1}{8(k+1)}\right).
\end{equation}
Notice that, since $M_2 \geq \max\{a(8N(k+1)-1), \ell(8N(k+1)-1)\}$
\begin{equation}\label{e:JBsmall}
\forall i \in [n',f(n')] \left(\max\{4N\alpha_i, 4N\lambda_i\} \leq \frac{1}{2(k+1)}\right).
\end{equation}
Hence, by \eqref{e:main5} and \eqref{e:main6}, for $i \in [n',f(n')]$
\begin{equation}\label{e:JABsmall}
 \norm{J^{\A}(y_{2i+1})-y_{2i+1}}\leq \frac{3}{4(k+1)} \quad \mbox{ and } \quad \norm{J^{\B}(y_{2i+2})-y_{2i+2}}\leq \frac{3}{4(k+1)}. 
\end{equation}
Using the fact that $J^\A$ and $J^\B$ are nonexpansive, we have  
\begin{equation*}
\begin{split}
&\norm{J^{\A}(y_{2i+2})-y_{2i+2}}\leq  2 \norm{y_{2i+2}-y_{2i+1}}+  \norm{J^\A (y_{2i+1})-y_{2i+1}}\\
\mbox{and} \qquad &\\
&\norm{J^{\B}(y_{2i+1})-y_{2i+1}}\leq  2 \norm{y_{2i+2}-y_{2i+1}}+  \norm{J^\B (y_{2i+2})-y_{2i+2}}.
\end{split}
\end{equation*}
By \eqref{e:normless8} and \eqref{e:JABsmall}, for all $i \in [n', f(n')]$ 
\begin{equation}\label{e:JABsmall2}
\norm{J^\A (y_{2i+2})-y_{2i+2}}\leq \frac{1}{k+1} \quad \mbox{ and } \quad \norm{J^\B (y_{2i+1})-y_{2i+1}}\leq \frac{1}{k+1}.
\end{equation}
The result follows from \eqref{e:JABsmall} and \eqref{e:JABsmall2} and the fact that $n' \leq \max\{\Psi_2(8k+7,n,\widetilde{f}[M_2]),M_2\}=\Psi_4(k,n,f)$. 
\end{proof}
The following result gives the quantitative version of the argument for the first case.
\begin{lemma} \label{l:case1}
For all $k,n \in \N$, monotone $f:\N \to \N$ and $x \in B_N$, if 
\begin{enumerate}[$(i)$]
	\item\label{lc13} $\forall m\in [n, \Psi_5(k,n,f)]\, \left( s_{m+1}^x\leq s_m^x \right)$
	\item\label{lc11} $\norm{J^\A(x)-x}\leq \dfrac{1}{\xi_1(n)+1} \wedge \norm{J^\B(x)-x}\leq \dfrac{1}{\xi_1(n)+1} $
	\item\label{lc12} $\forall y\in B_N\, \left(\left( \norm{J^\A(y)-y}\leq \frac{1}{n+1} \wedge  \norm{J^\B(y)-y}\leq \frac{1}{n+1}\right)\to \ip{ u-x, y-x} \leq \frac{1}{256(k+1)^2} \right)$
	\end{enumerate}
	then there exists $n' \leq \Psi_6(k,n,f)$ such that
	\begin{equation*}
	 \forall i\in[n', f(n')]\, \left(\norm{y_i-x}\leq \frac{1}{2(k+1)} \right).
	\end{equation*}    
\end{lemma}
\begin{proof}
We may assume that for all $m \leq \Psi_6(k,n,f)$ it holds that $f(m) \geq m$, otherwise the result is trivial.
By Lemma~\ref{l:asymregcase1} and the first hypothesis, there is $n_0 \leq \widetilde{n}$ such that for all $i  \in [n_0,j_f(n_0)]$
\begin{equation*}
 \norm{J^\A(y_{2i+1})-y_{2i+1}} \leq \frac{1}{n+1} \wedge  \norm{J^\B(y_{2i+1})-y_{2i+1}} \leq \frac{1}{n+1}
\end{equation*} 
and
\begin{equation*}
 \norm{J^\A(y_{2i+2})-y_{2i+2}} \leq \frac{1}{n+1} \wedge  \norm{J^\B(y_{2i+2})-y_{2i+2}} \leq \frac{1}{n+1}.
\end{equation*} 
Furthermore, from \eqref{e:normless8} in the proof of Lemma~\ref{l:asymregcase1} we also know that for $i \in [n_0,j_f(n_0)]$
\begin{equation}\label{e:eventoodd}
\norm{y_{2i+1}-y_{2i}}\leq \frac{1}{4(k+1)}.
\end{equation}
Hence, using the third hypothesis and the fact that $(y_m) \subset B_N$, for all $i  \in [n_0,j_f(n_0)]$
\begin{equation}\label{e:innerprodsmall}
\ip{ u-x, y_{2i+1}-x} \leq \frac{1}{256(k+1)^2} \wedge \ip{ u-x, y_{2i+2}-x} \leq \frac{1}{256(k+1)^2}.
\end{equation}
From \eqref{e:main3} we obtain
\begin{equation*}
\norm{y_{2i+2}-x}^2 \leq (1-\alpha_i)(1-\lambda_i)\left( \norm{y_{2i}-x}^2 + v_i\right)+ \alpha_i b_i+\lambda_ic_i
\end{equation*}
with $v_i := 8N\left(\norm{J^\A(x)-x}+\norm{J^\B(x)-x}\right)$, $b_i:=2\ip{ u-x,y_{2i+1}-x} +v_i$ and $c_i:= 2\ip{ u-x,y_{2i+2}-x} +v_i$.
By the second hypothesis and the monotonicity of $f$ and $\sigma_1$
\begin{equation}\label{e:errorsmall}
\begin{split}
 8N\left(\norm{J^\A(x)-x}+\norm{J^\B(x)-x}\right) &\leq \frac{16N}{1024N(k+1)^2(f(2\sigma_1(\widetilde{k},\widetilde{n})+1)+1)}\\
&\leq \frac{1}{64(k+1)^2(f(2\sigma_1(\widetilde{k},n')+1)+1)}.
\end{split}
\end{equation}
With $p=j_f(n_0)=f(2\sigma_1(\widetilde{k},n_0)+1)$, we have 
\begin{equation*}
\forall i  \in [n_0,p] \left( v_i \leq \frac{1}{64(k+1)^2(p+1)} \wedge b_i \leq \frac{1}{64(k+1)^2} \wedge c_i \leq \frac{1}{64(k+1)^2}\right),
\end{equation*}
using \eqref{e:innerprodsmall}, \eqref{e:errorsmall} and $p \geq \sigma_1(\widetilde{k},n_0) \geq 1$. By Lemma~\ref{l:Xuquant1} we conclude that
\begin{equation*}
\forall i \in [\sigma_1(\widetilde{k},n_0) ,p]\left(\norm{y_{2i}-x}\leq \frac{1}{4(k+1)}\right).
\end{equation*}
Since $\sigma_1(\widetilde{k},n_0)  \geq n_0$, we have that $ [\sigma_1(\widetilde{k},n_0) ,p] \subseteq [n_0,j_f(n_0)]$ and so by \eqref{e:eventoodd}
\begin{equation*}
\forall i \in [\sigma_1(\widetilde{k},n_0) ,p]\left(\norm{y_{2i+1}-x}\leq \frac{1}{2(k+1)}\right).
\end{equation*}
The result follows with $n':=2\sigma_1(\widetilde{k},n_0)+1$.
\end{proof}

\subsection{Second case}\label{s:Second}

We are now going to consider the case where the sequence $(s^{x}_n)$ is not eventually decreasing. In this subsection, we again  assume that $N\geq \max\{\norm{u-q}, \norm{x_0-q}, \norm{q}\}$, for some $q \in S$.

For $s:\N \to \N$ and $m \in \N$ we define a monotone function $\tau$ as follows.
\begin{equation*}
\tau^{s}_{m}(n):=
\begin{cases} n & n<m\\   
				\max\{ k \in [m,n]: s_{k}<s_{k+1}\} & n \geq m \wedge\exists k \in [m,n] \left(s_{k}<s_{k+1}\right)  \\
			    n   & n \geq m \wedge  \forall k \in [m,n] \left(s_{k+1}\leq s_{k}\right) . 
\end{cases}
\end{equation*}
\begin{lemma}[{\cite[Lemma~3.27]{DP(21)}}]\label{qtlemmaMainge}
Let $s:\N \to \N$ and $m,r \in \N$ be arbitrary. If $m \geq r$ and $s_{m}<s_{m+1}$, then
\begin{equation*}
 \forall i \geq m\left(\tau^{s}_{m}(i) \geq r \wedge \max\{s_{\tau^{s}_{m}(i)},s_i\}\leq s_{\tau^{s}_{m}(i)+1} \right) .
\end{equation*}
\end{lemma}
The following result gives the quantitative version of the argument for the second case.
\begin{lemma}\label{lemmamaincase2}
For $k,m,n \in \N$, $f:\N \to \N$  monotone and $x \in B_{N}$, assume that 
\begin{enumerate} [$(i)$]
\item\label{quantcase2} $n \geq r(m,k) \wedge s_n^{x} <s_{n+1}^{x}$,
\item\label{IneqJz-z} $\norm{J^{\A}(x)-x}\leq \frac{1}{\xi_2(n)+1} \wedge \norm{J^{\B}(x)-x}\leq \frac{1}{\xi_2(n)+1}$,
\item\label{hypprodint}
$\forall y \in B_{N} \left(\left(\norm{J^{\A}(y)-y}\leq \frac{1}{m+1} \wedge \norm{J^{\B}(y)-y}\leq \frac{1}{m+1}\right)\to \ip{ u-x,y-x } \leq \frac{1}{3 \cdot 96(k+1)^2} \right)$.
\end{enumerate}
Then, there exists $n' \leq 2n+1$ such that
$$\forall i \in [n',f(n')]\left(\norm{y_i -x} \leq \frac{1}{2(k+1)} \right).$$ 
\end{lemma}
\begin{proof}
	With $x\in B_N$, as in the statement of the lemma, we simplify the notation by writing $s_n:=s_n^x$, and $\tau_n:=\tau_n^{s^x}$. From Lemma~\ref{qtlemmaMainge} we know that
\begin{equation*}
\forall i \geq n \left( \tau_{n}(i) \geq r(m,k) \wedge  \max\{s_{\tau_{n}(i)},s_i\}\leq s_{\tau_{n}(i)+1} \right).
\end{equation*}
Let $i \in [n,f(2n+1)]$. Since $s_{\tau_{n}(i)} \leq s_{\tau_{n}(i)+1}$ and $r(m,k) \geq \max\{a(32\cdot64N^2(m+1)^2-1), \ell(24\cdot64N^2(m+1)^2-1)\}$, by \eqref{e:main4} we have
\begin{equation*}
\max\{\norm{y_{2{\tau_{n}(i)}+1}-y_{2{\tau_{n}(i)}}}^2,\norm{y_{2{\tau_{n}(i)}+2}-y_{2{\tau_{n}(i)}+1}}^2\} \leq \frac{16N^2}{32\cdot 64N^2(m+1)^2}+\frac{12N^2}{24\cdot 64N^2(m+1)^2}= \frac{1}{64(m+1)^2},
\end{equation*}
which entails
\begin{equation*}
\max\left\{\norm{y_{2{\tau_{n}(i)}+1}-y_{2{\tau_{n}(i)}}},\norm{y_{2{\tau_{n}(i)}+2}-y_{2{\tau_{n}(i)}+1}}\right\}\leq \frac{1}{8(m+1)}.
\end{equation*}
By the monotonicity of $a$ and $\ell$, we have ${\tau_{n}(i)} \geq r(m,k) \geq \max\{a(16N(m+1)-1), \ell(16N(m+1)-1)\}$, which implies 
\begin{equation*}
4N \alpha_{\tau_{n}(i)} \leq \frac{1}{4(m+1)} \quad \mbox{and} \quad 4N \lambda_{\tau_{n}(i)} \leq \frac{1}{4(m+1)}.
\end{equation*}
From \eqref{e:main5} and \eqref{e:main6} it follows that
\begin{equation*}
\norm{J^{\A}(y_{2{\tau_{n}(i)}+1})-y_{2{\tau_{n}(i)}+1}}\leq \frac{1}{2(m+1)} \wedge \norm{J^{\B}(y_{2{\tau_{n}(i)}+2})-y_{2{\tau_{n}(i)}+2}}\leq \frac{1}{2(m+1)}.
\end{equation*}
Hence
\begin{equation*}
 \norm{J^{\B}(y_{2{\tau_{n}(i)}+1})-y_{2{\tau_{n}(i)}+1}}\leq  2\norm{y_{2{\tau_{n}(i)}+2}-y_{2{\tau_{n}(i)}+1}}+ \norm{J^{\B}(y_{2{\tau_{n}(i)}+2})-y_{2{\tau_{n}(i)}+2}} \leq \frac{1}{m+1}
\end{equation*}
and
\begin{equation*}
 \norm{J^{\A}(y_{2{\tau_{n}(i)}+2})-y_{2{\tau_{n}(i)}+2}}\leq  2\norm{y_{2{\tau_{n}(i)}+2}-y_{2{\tau_{n}(i)}+1}}+ \norm{J^{\A}(y_{2{\tau_{n}(i)}+1})-y_{2{\tau_{n}(i)}+1}}\leq \frac{1}{m+1}.
\end{equation*}
Then, by the third hypothesis and the fact that $(y_n)\subset B_N$ (cf.~Lemma \ref{l:vbounded}),
\begin{equation}\label{e:prodintsmall}
\ip{ u-x,y_{2{\tau_{n}(i)}+1}-x } \leq \frac{1}{3 \cdot 96(k+1)^2} \quad \mbox{and} \quad \ip{ u-x,y_{2{\tau_{n}(i)}+2}-x } \leq \frac{1}{3 \cdot 96(k+1)^2}.
\end{equation}
Since $s_{\tau_{n}(i)} \leq s_{\tau_{n}(i)+1}$, it follows from \eqref{e:main3} that
\begin{equation*}
\begin{split}
(\alpha_{{\tau_{n}(i)}}+2\lambda_{{\tau_{n}(i)}}-\alpha_{{\tau_{n}(i)}}\lambda_{{\tau_{n}(i)}}) s_{{\tau_{n}(i)}+1} &\leq (\alpha_{{\tau_{n}(i)}} +2\lambda_{{\tau_{n}(i)}})(2\ip{ u- x, y_{2{\tau_{n}(i)}+1}-x} + \ip{ u- x, y_{2{\tau_{n}(i)}+2}-x}) \\
& \quad +8N\left(\norm{J^{\A}(x)-x}+\norm{J^{\B}(x)-x}\right), 
\end{split}
\end{equation*}
and, since $s_{\tau_n(i)+1}\leq 4N^2$ and $\alpha_{{\tau_{n}(i)}}\lambda_{{\tau_{n}(i)}}\leq \alpha_{{\tau_{n}(i)}}(\alpha_{{\tau_{n}(i)}}+2\lambda_{{\tau_{n}(i)}})$,
\begin{equation*}
\begin{split}
(\alpha_{{\tau_{n}(i)}}+2\lambda_{{\tau_{n}(i)}}) s_{{\tau_{n}(i)}+1} &\leq (\alpha_{{\tau_{n}(i)}} +2\lambda_{{\tau_{n}(i)}})(2\ip{ u- x, y_{2{\tau_{n}(i)}+1}-x} + \ip{ u- x, y_{2{\tau_{n}(i)}+2}-x}) \\
& \quad +8N\left(\norm{J^{\A}(x)-x}+\norm{J^{\B}(x)-x}\right) +4N^2\alpha_{{\tau_{n}(i)}}(\alpha_{{\tau_{n}(i)}}+2\lambda_{{\tau_{n}(i)}}). 
\end{split}
\end{equation*}
By the second hypothesis, $8N\left(\norm{J^{\A}(x)-x}+\norm{J^{\B}(x)-x}\right) \leq \dfrac{1}{96\mathcal{P}(f(2n+1))(k+1)^2}$ and so
\begin{equation*}
\begin{split}
s_{{\tau_{n}(i)}+1} &\leq 2\ip{ u- x, y_{2{\tau_{n}(i)}+1}-x} + \ip{ u- x, y_{2{\tau_{n}(i)}+2}-x}\\ 
& \quad + \frac{1}{\alpha_{{\tau_{n}(i)}}+\lambda_{{\tau_{n}(i)}}}\cdot\frac{1}{96\mathcal{P}(f(2n+1))(k+1)^2} + 4N^2\alpha_{{\tau_{n}(i)}}. 
\end{split}
\end{equation*}
We have $\tau_n(i) \leq i \leq f(2n+1)$. Hence by \eqref{e:prodintsmall} and the monotonicity of $\mathcal{P}$,
\begin{equation*}
s_{{\tau_{n}(i)}+1}\leq  \frac{2}{96(k+1)^2} + 4N^2\alpha_{{\tau_{n}(i)}}. 
\end{equation*}
By the monotonicity of $a$, we have $r(m,k) \geq a(4\cdot96N^2(k+1)^2-1)$, and thus 
\begin{equation*}
4N^2\alpha_{{\tau_{n}(i)}} \leq \frac{1}{96(k+1)^2}.
\end{equation*}
We conclude that $s_{{\tau_{n}(i)}+1} \leq \dfrac{1}{32(k+1)^2}$, and so $s_i \leq \dfrac{1}{32(k+1)^2}$. Consequently, 
\begin{equation}\label{e:case2_even}
	\forall i\in [n, f(2n+1)]\, \left(\|y_{2i}-x\|\leq \frac{1}{4(k+1)}\right).
\end{equation}
By \eqref{e:main4} and the fact that $s_{i+1}\geq 0$, we conclude
\begin{equation*}
		\norm{y_{2i+1}-y_{2i}}^2 \leq 16N^2\alpha_i  + 12N^2\lambda_i + \frac{1}{32(k+1)^2}.
\end{equation*}
Since $i\geq n \geq r(m,k)\geq \max\{a(16\cdot64N^2(k+1)^2-1), \ell(12\cdot64N^2(k+1)^2-1)\}$, it follows $\|y_{2i+1}-y_{2i}\|\leq \frac{1}{4(k+1)}$.
Hence, using \eqref{e:case2_even}
\begin{equation}\label{e:case2_odd}
	\forall i\in [n, f(2n+1)]\, \left(\|y_{2i+1}-x\|\leq \frac{1}{2(k+1)}\right).
\end{equation}
Now let $i\in [2n+1, f(2n+1)]$. If $i=2i'$, then $i'\in[n, f(2n+1)]$ and the result follows from \eqref{e:case2_even}. If $i=2i'+1$, then the result follows from \eqref{e:case2_odd} as again $i'\in[n, f(2n+1)]$. This concludes the proof.
\end{proof}
\subsection{Putting it together}\label{s:Together}
%
%
%
%

We are now able to prove the main result of this section.
\begin{proof}[Proof of Theorem~\ref{t:quantBM2}]
	Let $k\in \N$ and a monotone function $f:\N\to\N$ be given. Since $N \geq 2 \norm{u-q}$, by Lemma~\ref{l:projection}, there exist $n_0\leq \zeta(\overline{k},\Xi)$ and $x\in B_N$ such that
\begin{equation*}
	\|J^{\A}(x)-x\|\leq \frac{1}{\Xi(n_0)+1} \land \|J^{\B}(x)-x\|\leq \frac{1}{\Xi(n_0)+1}
\end{equation*}
	and
\begin{equation*}
	\forall y\in B_N\, \left(\left(\|J^{\A}(y)-y\|\leq \frac{1}{n_0+1} \wedge \|J^{\B}(y)-y\|\leq \frac{1}{n_0+1}\right)\to \ip{ u-x, y-x}\leq \frac{1}{3\cdot 96(k+1)^2}\right).
\end{equation*}
	Assume that for all $m\in [r^{+}(n_0,k), \Psi_5(k, r^{+}(n_0, k),f)]$ it holds that $s_{m+1}\leq s_m$. Notice that $\Xi(n_0)\geq \Xi_1(n_0)=\xi_1(r^{+}(n_0,k))$ and $r^{+}(n_0,k)\geq n_0$. Applying Lemma~\ref{l:case1} (with $n=r^{+}(n_0,k)$) we conclude that
\begin{equation}\label{e:conclusioncase1}
	\exists n\leq \Psi_6(k,r^{+}(n_0,k),f)\, \forall i\in [n, f(n)] \left(\|y_i-x\|\leq \frac{1}{2(k+1)}\right).
\end{equation}
	Now assume that $s_{n_1}<s_{n_1+1}$, for some $n_1\in[r^{+}(n_0,k), \Psi_5(k, r^{+}(n_0, k),f)]$. Since $\xi_2$ is monotone, we have $\Xi(n_0)\geq \Xi_2(n_0)=\xi_2(\Psi_5(k, r^{+}(n_0, k),f))\geq \xi_2(n_1)$. Then, applying Lemma~\ref{lemmamaincase2} (with $m=n_0$ and $n=n_1$) we conclude that
\begin{equation}\label{e:conclusioncase2}
\exists n\leq 2n_1+1 \,\forall i\in [n, f(n)] \left(\|y_i-x\|\leq \frac{1}{2(k+1)}\right).
\end{equation}
	 Hence, there is $n\leq \mu(k,f)$ such that for all $i,j\in[n, f(n)]$,
	 \[\|y_i-y_j\|\leq \|y_i-x\|+\|y_j-x\|\leq \frac{1}{k+1}.\qedhere\]
\end{proof}
\begin{remark}
The proof of Theorem~\ref{t:quantBM2} actually shows that for all $k \in \N$ and for all $f:\N \to \N$
\begin{equation*}
\exists n \leq \mu(k,f)\,\exists x \in B_{N}\, \forall m \in [n,f(n)] \left(\norm{y_m-x}\leq\frac{1}{2(k+1)} \wedge \norm{J^\A(x)-x} \leq \frac{1}{k+1} \wedge \norm{J^\B(x)-x} \leq \frac{1}{k+1}\right) .
\end{equation*}
\end{remark}
%
\section{A generalization with error terms}\label{s:Reduction}
In this section we show that $\norm{x_n-y_n} \to 0$ and compute rates of convergence. These rates entail that the main results of Sections~\ref{s:3} and \ref{s:second} can be extended from \eqref{HPPA2} to \eqref{HPPA2e}.
\begin{lemma}\label{l:rate1}
Assume that $\sum \alpha_n= \infty$ $($or $\sum \lambda_n= \infty)$ with rate of divergence ${\rm A}$, and $\sum \norm{e_n} < \infty$ and $\sum \norm{e'_n} <\infty$ with Cauchy rates ${\rm E}, {\rm E'}$, respectively. Consider $M \in \N$ such that $M\geq  \sum_{i=0}^{{\rm E}(1)}\norm{e_n}+\sum_{i=0}^{{\rm E'}(1)}\norm{e'_n}+1$. Then $\norm{x_n-y_n} \to 0$ with rate of convergence
 $$\rho(k):=\rho[{\rm A}, {\rm E}, {\rm E'}, M](k):=\max\{2 \rho_1(2k+1)+1, 2{\rm E}(2k+1)+3\},$$
where $\rho_1(k)=\rho_1[{\rm A},\mathbf{0},\mathbf{0},{\rm D},M](k)$ is as in Lemma~\ref{l:Xuquant3},
with ${\rm D}(k):=\max\{{\rm E}(2k+1), {\rm E'}(2k+1) \}$ and $\mathbf{0}$ denotes the zero function.
\end{lemma}
\begin{proof}
Since the resolvent functions are nonexpansive, we have that 
\begin{equation}\label{e:x-yodd}
\norm{x_{2n+1}-y_{2n+1}} \leq (1-\alpha_n)\norm{x_{2n}-y_{2n}}+\norm{e_n}
\end{equation}
and
\begin{equation}\label{e:x-yeven}
\norm{x_{2n+2}-y_{2n+2}} \leq (1-\lambda_n)\norm{x_{2n+1}-y_{2n+1}}+\norm{e'_n}.
\end{equation}
Observe  that $\sum_{i=0}^n \norm{e_i}+\sum_{i=0}^n \norm{e'_i}\leq M$. By induction, using \eqref{e:x-yodd} and \eqref{e:x-yeven}, one shows that $\norm{x_n-y_n}\leq M$, for all $n \in \N$. Also from \eqref{e:x-yodd} and \eqref{e:x-yeven} one derives
\begin{equation*}
\norm{x_{2n+2}-y_{2n+2}} \leq (1-\alpha_n)(1-\lambda_n)\norm{x_{2n}-y_{2n}}+\norm{e'_n}+\norm{e_n}.
\end{equation*}
It is easy to see that ${\rm D}$ is a Cauchy rate for $(\sum( \norm{e_n}+\norm{e'_n}))$. Hence by Lemma~\ref{l:Xuquant3}
\begin{equation*}
\forall k \in \N \, \forall n \geq \rho_1[{\rm A},\mathbf{0},\mathbf{0},{\rm D},M](2k+1) \left( \norm{x_{2n}-y_{2n}} \leq \frac{1}{2(k+1)}\right).
\end{equation*}
Since, for $n \geq {\rm E}(2k+1)+1$, it holds that $\norm{e'_n}\leq \frac{1}{2(k+1)}$. The result now follows from  \eqref{e:x-yodd}.
\end{proof}
In a similar way one can obtain a rate of convergence, using Lemma~\ref{l:Xuquant4} instead.
\begin{lemma}\label{l:rate2}
Assume that $\prod (1-\alpha_n)= 0$ (or $\prod (1-\lambda_n)=0$) with a function ${\rm A'}$ satisfying \eqref{q2'}, and $\sum \norm{e_n} < \infty$ and $\sum \norm{e'_n} <\infty$ with Cauchy rates ${\rm E}$ and ${\rm E'}$ respectively. Consider $M \in \N$ such that $M\geq \sum_{i=0}^{{\rm E}(1)}\norm{e_n}+\sum_{i=0}^{{\rm E'}(1)}\norm{e'_n}+1$. Then $ \norm{x_n-y_n} \to 0$ with rate of convergence
 $$\rho(k):=\rho[{\rm A'}, {\rm E}, {\rm E'}, M](k):=\max\{2 \rho_2(2k+1)+1, 2{\rm E}(2k+1)+3\},$$
where $\rho_2(k)=\rho_2[{\rm A'},\mathbf{0},\mathbf{0},{\rm D},M](k)$ is as in Lemma~\ref{l:Xuquant4},
with ${\rm D}(k):=\max\{{\rm E}(2k+1), {\rm E'}(2k+1) \}$ and $\mathbf{0}$ denotes the zero function.
\end{lemma}
\begin{remark}\label{r:errors}
In \cite[Theorem~3.2]{BM(13)} and \cite[Theorem~8]{BM(12)} several other conditions on the error terms are considered. Under those conditions, one can obtain rates of convergence for $\norm{x_n -y_n} \to 0$ by following the steps in the proofs of Lemmas~\ref{l:rate1} and \ref{l:rate2}. Indeed, the only difference is in the choice of parameters for the functions $\rho_1$ and $\rho_2$. For example, under the condition $\sum \norm{e_n}< \infty$ and $\norm{e'_n}/\alpha_n \to 0$, one would assume the existence of a Cauchy rate ${\rm E}$ for  $(\sum \norm{e_n})$ as before, ${\rm B}$ a rate of convergence for $\norm{e'_n}/\alpha_n \to 0$, and consider $M' \geq \max_{i < {\rm B}(0)}\left\{1,\norm{e'_i}/\alpha_i\right\}+\sum_{i=0}^{{\rm E}(0)}\norm{e_i}+1 $. Then $\rho(k)$ is defined as in Lemmas~\ref{l:rate1} and \ref{l:rate2} but instead with (respectively)
\begin{equation*}
\rho_1[{\rm A},{\rm B}, \mathbf{0},{\rm E},M']   \mbox{\quad and \quad } \rho_2[{\rm A'},{\rm B}, \mathbf{0},{\rm E},M'].
\end{equation*} 
\end{remark}
\begin{remark}\label{r:HPPAalte}
Using the same arguments as in Lemmas~\ref{l:rate1} and \ref{l:rate2} (and Remark~\ref{r:errors}) one shows that $\rho$ is also a rate of convergence for $\norm{w_n-z_n} \to 0$, where $(z_n)$ is generated by \eqref{HPPAalt} and $(w_n)$ is given by the generalized algorithm with error terms \eqref{HPPAalte}, i.e.\ the following  iteration:
\begin{equation}\label{HPPAalte}\tag{\textsf{HPPA}$_2$}
\begin{cases}
w_{2n+1}= \widetilde{\alpha}_n u+(1-\widetilde{\alpha}_n) J_{\widetilde{\beta}_n}^{\A'} \left(w_{2n}\right) +e_n \\ 
w_{2n+2}= \widetilde{\lambda}_n u+(1-\widetilde{\lambda}_n) J_{\widetilde{\mu}_n}^{\B'} \left(w_{2n+1}\right) + e'_n
\end{cases}
\end{equation}
where $w_0=z_0 \in H$ is given, $(\widetilde{\alpha}_n), (\widetilde{\lambda}_n) \subset (0,1) $, $(\widetilde{\beta}_n),(\widetilde{\mu}_n) \subset (0,+\infty)$. 
\end{remark}
\section{Final remarks}\label{s:final}

In this final section we discuss in what way our results establish a finitization of the strong convergence of the method of alternating resolvents towards a common zero.

Section~\ref{s:3} gives an effective metastability bound for $(y_n)$ given by the algorithm \eqref{HPPA2}. Recall that the restriction to monotone functions and to the interval $[n,f(n)]$ are only for convenience as explained in Remark~\ref{r:maj}. We begin by computing a partial bound on the metastability of the iteration which depends on the  conditions $(Q_3)-(Q_5)$ together with $(Q_{\eta})$ and $(Q_{\eta'})$. In Subsection~\ref{s:qrar}, using the conditions $(Q_1),(Q_2),(Q_4),(Q_6)$ and $(Q_7)$, we compute functions $\eta$ and $\eta'$ satisfying conditions $(Q_{\eta})$ and $(Q_{\eta'})$ and thus obtain a metastability bound depending only on the conditions $(Q_1)-(Q_7)$ (cf. Remark~\ref{r:eta}).
 From the metastability of the iteration $(y_n)$ it follows (ineffectively) that $(y_n)$ is a Cauchy sequence. Hence it converges strongly to some point $y \in H$. From Remark~\ref{r:eta} and the continuity of $J^{\A}$ and $J^{\B}$, we conclude that $y$ must be a common zero of the operators $\A$ and $\B$. Furthermore, one can argue that $y$ must be the projection point of $u$ onto $S$. Consider $\widetilde{y}$ such projection point. Since $y_n \to y \in S$, we have $\ip{ u-\widetilde{y},y-\widetilde{y}} \leq 0$ and thus 
\begin{equation*}
\forall k \in \N \, \exists n \in \N \, \forall m \geq n \left(\ip{ u-\widetilde{y},y_m-\widetilde{y}} \leq \frac{1}{k+1} \right).
\end{equation*}
Hence, for all $k$ and $f$, \eqref{e:almostzero} and \eqref{e:prodintsmalli} (in the proof of Theorem~\ref{t:quantBM}) hold with $x_0=\widetilde{y}$, and $n_0$ big enough. Following the proof of Theorem~\ref{t:quantBM} we conclude that $y_{2m} \to \widetilde{y}$ and therefore $y=\widetilde{y}$.

Section~\ref{s:second} studies the metastability of $(y_n)$ given by \eqref{HPPA2}, under different assumptions. Here one drops the conditions $(Q_5),(Q_6)$ and $(Q_7)$. It is however necessary to consider additionally the condition $(Q_8)$. By Theorem~\ref{t:quantBM2} the sequence  $(y_n)$ is metastable and hence convergent to some point $y \in H$. From Corollary~\ref{c:metatoasymp} and the continuity of  $J^{\A}$ and $J^{\B}$, such point $y$ must be a common zero of the operators $\A$ and $\B$. Let  $\widetilde{y}$ be the projection point of $u$ onto $S$. Then the conclusions of Lemmas~\ref{l:case1} and \ref{lemmamaincase2} hold with $x=\widetilde{y}$, which implies $y= \widetilde{y}$. Notice that one cannot guarantee the third assumption in neither of those lemmas. However, those conditions are only required to show equations \eqref{e:innerprodsmall} and \eqref{e:prodintsmall}, respectively, which follow from the fact that $\ip{ u- \widetilde{y}, y-\widetilde{y}} \leq 0$ and the fact that $y_n \to y$. Indeed, let $k$ and monotone $f$ be given. For $n_0$ big enough, we have $\ip{ u-\widetilde{y}, y_{2n+1}-\widetilde{y} \,} \leq \frac{1}{3 \cdot 96(k+1)^2}$ and $\ip{u-\widetilde{y}, y_{2n+2}-\widetilde{y}} \leq \frac{1}{3 \cdot 96(k+1)^2}$, for $n \geq n_0$. If $s^{\widetilde{y}}_{m+1} \leq s^{\widetilde{y}}_m$ for all $m \in [r^{+}(n_0,k), \Psi_5(k,r^{+}(n_0,k),f)]$, the conclusion of Lemma~\ref{l:case1} holds with $n =r^{+}(n_0,k)$ by following its proof after \eqref{e:innerprodsmall} (with $x=\widetilde{y}$). If $s^{\widetilde{y}}_{n_1}<s^{\widetilde{y}}_{n_1+1}$, for some $n_1\in[r^{+}(n_0,k), \Psi_5(k, r^{+}(n_0, k),f)]$, then the conclusion of Lemma~\ref{lemmamaincase2} holds with $n=n_1$. In fact, since for $i \geq n_1$ one has $\tau_{n_1}(i) \geq n_1 \geq r^{+}(n_0,k) \geq n_0$, we have \eqref{e:prodintsmall} (with $x=\widetilde{y}$) and can follow the proof of Lemma~\ref{lemmamaincase2} (with $m=n_0$) after that point. This argument shows 
\begin{equation*}
\forall k \in \N\, \forall f :\N \to \N \, \exists n \, \forall i \in [n,f(n)] \left( \norm{y_{i}-\widetilde{y}}\leq \frac{1}{k+1} \right),
\end{equation*}
implying $y_n \to \widetilde{y}$ and consequently $y= \widetilde{y}$.

Finally, we argue that Section~\ref{s:Reduction} extends the results from Sections~\ref{s:3} and \ref{s:second} to an iteration $(x_n)$ given by \eqref{HPPA2e}, which is a generalized version of the algorithm \eqref{HPPA2} that allows error terms. 
Lemmas~\ref{l:rate1} and \ref{l:rate2} give rates of convergence for $\norm{x_n-y_n} \to 0$. As such, we also have that  $(x_n)$ converges to the projection point of $u$ onto $S$. Moreover, one has a rate of metastability for $(x_n)$. Indeed, let $\rho$ be a rate of convergence for $\norm{x_n-y_n} \to 0$ (as in Lemmas~\ref{l:rate1} or \ref{l:rate2}) and $\mu$ be a rate of metastability for $(y_n)$ (as in Corollary~\ref{c:quantBM} -- together with Remark~\ref{r:eta} -- or Theorem~\ref{t:quantBM2}). Given $k \in \N$ and monotone $f: \N \to \N$, there exists $n_0 \leq \mu(2k+1, \widetilde{f}[\rho(4k+3)])$ such that 
\begin{equation*}
\forall i,j \in [n_0,\widetilde{f}[\rho(4k+3)](n_0)] \left( \norm{y_i-y_j}\leq \frac{1}{2(k+1)}\right).
\end{equation*}
With $n:= \max\{n_0,\rho(4k+3)\}$, we have for all $i,j \in [n,f(n)]$
\begin{equation*}
\norm{x_i-x_j}\leq \norm{x_i-y_i}+\norm{y_i-y_j}+\norm{x_j-y_j}\leq \frac{1}{4(k+1)}+\frac{1}{2(k+1)} +\frac{1}{4(k+1)} \leq \frac{1}{k+1},
\end{equation*}
and so $\max\{\mu(2k+1, \widetilde{f}[\rho(4k+3)]),\rho(4k+3)\}$ is a rate of metastability for $(x_n)$. In a similar way one can extend the quasi-rate of asymptotic regularity for \eqref{HPPA2} (cf. Remark~\ref{r:eta} and Corollary~\ref{c:metatoasymp}) to a quasi-rate of asymptotic regularity for \eqref{HPPA2e}.
The same constructions also hold for the Halpern-type algorithm \eqref{HPPAalte}. In light of Proposition~\ref{MarHalpern} and Remarks~\ref{r:rar} and \ref{r:HPPAalte}, one also obtains a rate of metastability and a quasi-rate of asymptotic regularity for $(w_n)$ generated by \eqref{HPPAalte}.

Altogether our results give a finitization of the proofs of \cite[Theorem~3.2]{BM(13)} (Theorem~\ref{t:BM_exact}) and \cite[Theorem~8]{BM(12)}. Moreover, our quantitative results avoid the use of the metric projection (relying instead on Lemma~\ref{l:projection}) and bypass the sequential weak compactness argument used in the original proofs, applying the macro introduced in \cite[Proposition~4.3]{FFLLPP(19)}.

\section*{Acknowledgements}

\quad \, Both authors acknowledge the support of FCT - Funda\c{c}\~ao para a Ci\^{e}ncia e Tecnologia under the projects: UIDB/04561/2020 and UIDP/04561/2020, and the research center Centro de Matem\'{a}tica, Aplica\c{c}\~{o}es Fundamentais e Investiga\c{c}\~{a}o Operacional, Universidade de Lisboa. 

The second author was also supported by the German Science Foundation (DFG Project KO 1737/6-1).

\bibliography{References}{}
\bibliographystyle{plain}

\end{document}